\documentclass[11pt]{amsart}
\usepackage{tikz-cd}
\usepackage{amsmath}
\usepackage{amssymb}
\usepackage{amsthm}
\usepackage{latexsym}
\usepackage{graphicx}
\usepackage{hyperref}
\usepackage{enumerate}
\usepackage[all]{xy}
\usepackage{amsfonts}
\usepackage{amstext}
\usepackage{mathrsfs}
\usepackage{lingmacros}
\usepackage{pdfsync}
\usepackage[colorinlistoftodos]{todonotes}
\usepackage{imakeidx}
\usepackage{hyperref}
\hypersetup{
  colorlinks   = true,    
  urlcolor     = blue,    
  linkcolor    = blue,    
  citecolor    = red      
}

\newtheorem{thm}{Theorem}[section]
\newtheorem{cor}[thm]{Corollary}
\newtheorem{lem}[thm]{Lemma}
\newtheorem*{lem*}{Lemma}
\newtheorem{prop}[thm]{Proposition}
\newtheorem{defn}[thm]{Definition}

\newtheorem{notation}[thm]{Notation}
\newtheorem{rem}[thm]{Remark}
\newtheorem{corintro}{Corollary}
\newtheorem{thmintro}{Theorem}

\newcommand{\bbA}{\mathbb{A}}
\newcommand{\bbC}{\mathbb{C}}

\newcommand{\bbG}{\mathbb{G}}
\newcommand{\bbH}{\mathbb{H}}
\newcommand{\bbN}{\mathbb{N}}
\newcommand{\bbR}{\mathbb{R}}
\newcommand{\bbZ}{\mathbb{Z}}
\newcommand{\bo}{\mathbf{o}}

\newcommand{\cC}{\mathcal{C}}
\newcommand{\cE}{\mathcal{E}}
\newcommand{\cF}{\mathcal{F}}
\newcommand{\cH}{\mathcal{H}}
\newcommand{\cL}{\mathcal{L}}
\newcommand{\cO}{\mathcal{O}}
\newcommand{\cP}{\mathcal{P}}
\newcommand{\cR}{\mathcal{R}}

\newcommand{\cX}{\mathcal{X}}
\newcommand{\cY}{\mathcal{Y}}
\newcommand{\fa}{\mathfrak{a}}
\newcommand{\fb}{\mathfrak{b}}

\newcommand{\fd}{\mathfrak{d}}

\newcommand{\fg}{\mathfrak{g}}

\newcommand{\fn}{\mathfrak{n}}
\newcommand{\fo}{\mathfrak{o}}
\newcommand{\ft}{\mathfrak{t}}
\newcommand{\fF}{\mathfrak{F}}
\newcommand{\fT}{\mathfrak{T}}
\newcommand{\fR}{\mathfrak{R}}
\newcommand{\fX}{\mathfrak{X}}

\newcommand{\sr}{\mathsf{r}}
\newcommand{\dB}{\hat{B}}
\newcommand{\db}{\hat{\fb}}
\newcommand{\dG}{\hat{G}}
\newcommand{\dg}{\hat{\fg}}
\newcommand{\dn}{\hat{\fn}}
\newcommand{\dT}{\hat{T}}
\newcommand{\dt}{\hat{\ft}}
\newcommand{\dHc}{\hat{H}_\chi}
\newcommand{\ap}{\left[{\textup{AP}([T,\cO(\chi)])}\right]}
\newcommand{\preap}{\textup{AP}([T,\cO(\chi)])}
\newcommand{\aq}{G(F)Z_G(\bbA_F)\backslash G(\bbA_F)}
\newcommand{\Ltwoxi}{L^2(G(F)\backslash G(\bbA_F),\xi)}
\newcommand{\rA}{{}^r\!A}
\newcommand{\rR}{\hat{R}}
\newcommand{\ra}{\hat{\alpha}}
\newcommand{\rb}{\hat{\beta}}
\newcommand{\id}{\mathop{\hbox{\rm id}}\nolimits}
\newcommand{\cCO}{\mathcal{C}(\Omega)}
\DeclareMathOperator{\Gal}{Gal}
\DeclareMathOperator{\res}{res}
\DeclareMathOperator{\Hom}{Hom}
\DeclareMathOperator{\Lie}{Lie}
\DeclareMathOperator{\Val}{Val}
\DeclareMathOperator{\re}{Re}
\DeclareMathOperator{\im}{Im}
\DeclareMathOperator{\Res}{Res}
\DeclareMathOperator{\codim}{codim}
\DeclareMathOperator{\supp}{supp}
\makeindex[columns=2]

\begin{document}
\title{On the unramified spherical automorphic spectrum} 

\author{Marcelo De Martino}
\address{M.D.M: Department of Eletronics and Informations Systems\\
Clifford Research Group\\
9000 Ghent\\
Belgium\\
email: marcelo.goncalvesdemartino@ugent.be}
\author{Volker Heiermann}
\address{
V.H.: Aix Marseille Universit\'e\\
CNRS \\
Centrale Marseille\\
I2M UMR 7373\\
13453\\ 
Marseille\\
France\\
email: volker.heiermann@univ-amu.fr}
\author{Eric Opdam}
\address{E.O.: Korteweg de Vries Institute for Mathematics\\
University of Amsterdam\\
P.O. Box 94248\\
1090 GE Amsterdam\\
The Netherlands\\
 email: e.m.opdam@uva.nl}


\begin{abstract}  
For an unramified  connected reductive group $G$ defined over a number field $F$, consider the part of the spherical automorphic spectrum with cuspidal support $[T,\cO(\chi)]$, where $T$ is a maximal torus and $\chi$ is an unramified automorphic character. We define a normalization of the Eisenstein series and we give the precise spectral decomposition of the closure of the subspace spanned by the normalized pseudo-Eiseinstein series. The proof uses residue distributions which were introduced by the third author (in joint work with G. Heckman) in the study of graded affine Hecke algebras, which is an ingredient of a purely local nature. 

In the case when $G$ is split and $\chi$ is the trivial character, we show that the normalized spectrum is in fact the whole spherical automorphic spectrum. The necessary argument to conclude the result in the split case are based on combinatorial results proved in \cite{DMHO}.
\end{abstract}

\maketitle

\section{Statement of the results}

\subsection{On reductive groups}
Consider the triple $(G,B,T)$, where $G$\index{$G$} is a connected reductive quasi-split algebraic group, defined over a number field $F$, $B=TU$ is an $F$-Borel subgroup, with $T$\index{$G$!$T$} a maximal torus in $B$\index{$G$!$B$} and $U$\index{$G$!$U$} the unipotent radical of $B$.  We let $A$\index{$G$!$A$} denote the maximal $F$-split torus in $T$ and we let $E$ be a minimal Galois extension of $F$ which is unramified
in each non-Archimedean place (we do not assume that
the Archimedean places split). Put $\Gamma := \Gal(E/F)$\index{$\Gamma$}. We let $Z_{G}$\index{$G$!$Z_G$} be the centre of $G$.

Given any reductive group $H$ over $F$, we let $X^*(H)$ and $X_*(H)$ denote the lattices of (absolute) characters and cocharacters of $H$, respectively. For $H = T$, both lattices are endowed with natural actions of the Galois group $\Gamma$. We will use a $\Gamma$-superscript to denote the invariants for the $\Gamma$-action. Consider the real vector space $\fa := \bbR\otimes X_*(T)^\Gamma = \bbR\otimes X_*(A)$\index{$\fa$}, which is viewed as the real Lie algebra of $A$, and its dual $\fa^* := \bbR\otimes X^*(T)^\Gamma$\index{$\fa$!$\fa^*$}. We let $R:=R(G,A)\subseteq X^*(A)$\index{Root Systems!$R$ relative roots} be the relative root system of $G$ and $W_G := W(G,A)$\index{Weyl Groups!$W_G$} be its Weyl group (remark that $G(F)$ contains representatives for $W_G$ as $W_G \cong \textup{Norm}_{G(F)}T(F)/\textup{Cent}_{G(F)}T(F))$. We let also $\tilde{R}\subseteq X^*(T)$\index{Root Systems!$\tilde{R}$ absolute roots} be the absolute root system and denote its elements by $\tilde\alpha$. We recall that, for every $\alpha\in R$, we can write $\alpha = \res(\tilde\alpha)$, where  $\res:X^*(T)\to X^*(A)$ is the restriction map coming from the inclusion $A\subseteq T$. The subspace of $\fa^*$\index{$\fa$!$\fa^*$} spanned by the relative roots will be denoted $\fa^{G*}\subseteq \fa^*$\index{$\fa$!$\fa^{G*}$} and its complexification will be denoted by $\fa_\bbC^{G*}\subseteq \fa^*_\bbC$.

\subsection{Unramified characters and the automorphic space}

We will denote by $\bbA_F$\index{$\bbA_F$} and $\bbA_E$\index{$\bbA_F$!$\bbA_E$} the ring of ad\`eles with respect to the fields $F$ and $E$. 
Given $\lambda \in X^*(T)$, let $|\lambda|$ denote the adelic absolue value of $\lambda$ and denote by $T^1 = \cap \{\ker |\lambda|\mid\lambda\in X^*(T)\}$. Following \cite[I.1.4]{MW2}, define $X_T$\index{$X_T$} to be the group of all continuous homomorphisms $T(\bbA_F)\to \bbC^\times$ which are trivial on $T^1$. Since $T(\bbA_F)/T^1 \cong \Val_F\otimes X_*(T)^\Gamma$, and $\Val_F = \bbR_+ = e^\bbR$ (recall we assume $F$ to be a number field),
if $\dT = \Hom(X_*(T),\bbC^\times)$\index{$G$!$\dT$} denotes the dual torus, we can identify $\fa_\bbC^* = \Lie(\dT)^\Gamma$ and
the group $X_T$ satisfy $X_T \cong \Lie(\dT)^\Gamma$.
 Furthermore, denote by $X^G_T\subseteq X_T$\index{$X_T$!$X_T^G$} the subgroup consisting of all elements of $X_T$ which vanish on $Z_G(\bbA_F)$, so that $X_T^G\cong \fa_\bbC^{G*}$.

Let $\chi = \otimes_v\chi_v$\index{$\chi$} be a unitary, everywhere unramified automorphic character of $T$ and denote by $\xi=\chi|_{Z_G}$\index{$\chi$!$\xi$} its restriction to the centre. We let $\Ltwoxi$ be the space of all functions on $G(F)\backslash G(\bbA_F)$ on which $Z_G(\bbA_F)$ acts by $\xi$ and that are square integrable modulo the centre. In this paper, we are interested in spherical automorphic forms which are supported by a cuspidal datum of the type
 $\fX =[T, \cO(\chi)]$\index{$[T,\chi]$} (for the decomposition of the automorphic space in terms of cuspidal datum we refer the reader to \cite{La1} and \cite{MW2}), where $\cO(\chi) = \{\chi_\nu:=\chi\otimes \nu\mid \nu\in X_T^G\}$\index{$\cO(\chi)$} denotes its orbit of unramified character twists. For each place $v$ of $F$, we let $K_v$\index{$G$!$K_v$} be a maximal compact subgroup of $G_v$\index{$G$!$G_v$} with the property that, at each non-Archimedean places, $K_v$ is a hyperspecial compact subgroup of $G_v$. Let $K :=\prod_v K_v$\index{$G$!$K$} be the corresponding maximal compact subgroup of $ G(\bbA_F)$\index{$G$!$G(\bbA_F)$}.  The main object of study in this manuscript is the space\index{$\Ltwoxi_{[T,\cO(\chi)]}^K$}
\begin{equation}\label{eq:sphspace}
\Ltwoxi_{[T,\cO(\chi)]}^K,
\end{equation}
of $K$-fixed vectors inside the space of automorphic forms supported by $[T,\cO(\chi)]$. It is known that this space is topologically generated by the so-called pseudo-Eisenstein series, $\theta_\phi$, which are very particular functions in $L^2(G(F)\backslash G(\bbA_F),\xi)$ given as wave-packets of Eisenstein series with Paley-Wiener coefficients. Below, in Section \ref{s:NormEisen}, we will recall the definitions.

\subsection{The spectral side}

We shall phrase our results in terms of Arthur parameters. 
Let $\dG$\index{$G$!$\dG$} be the complex dual group of $G$ and $\dT = \Hom(X_*(T),\bbC^\times)$\index{$G$!$\dT$} be the dual torus, which has a decomposition $\dT= \dT_{\textup{rs}}\dT_{\textup{u}}$ into its real split and unitary parts. Their complex Lie algebras will be denoted by $\dg$\index{$\dg$} and $\dt$\index{$\dg$!$\dt$}, respectively. Denote by ${}^LG = \dG\rtimes \Gamma$\index{$G$!${}^LG$} the (finite form of the) $L$-group of $G$ and denote by ${}^LT= \dT\rtimes \Gamma$\index{$G$!${}^LT$} the $L$-torus. Let  $W_{E/F}$\index{$W_{E/F}$ relative Weil group} be the relative Weil group. From \cite{La2}, let $\varphi_\chi\in H^1_{\textup{cont}}(W_{E/F},\dT_{\textup{u}})$ be a (unitary) Langlands parameter associated with $\chi$. For any character $\nu\in X_T^G$, we let $\varphi_{\chi,\nu}=\varphi_\chi\varphi_\nu$ denote the Langlands parameter of the twisted character $\chi_\nu$. The twisted parameter $\varphi_{\chi,\nu}$ is unitary if and only if $\nu$ is unitary, that is, $\nu$ is in $i\fa^{*}$.
\begin{defn}\label{d:APspace}
Denote by $\preap$\index{$\ap$!$\preap$} the space of continuous homomorphisms $\psi:W_{E/F}\times SL(2,\bbC)\to{}^LG$ satisfying
\begin{itemize}
\item[(a)] $\rho:=\psi|_{SL(2,\bbC)}$ is algebraic (and its image lies in $\dG$),
\item[(b)] for all $w\in W_{E/F}$, the projection of  $\psi(w)$ onto $\Gamma$ is the same as the image of $w$ by the map $W_{E/F}\to \Gamma$,
\item[(c)] $\psi|_{W_{E/F}}$ is bounded and conjugate to $\varphi_{\chi,\nu}$ for some unitary $\nu\in X_T^G.$
\end{itemize}
The dual group $\dG$ acts on this space by conjugation and let\index{$\ap$}
\begin{equation}\label{eq:ArthurParSpace}
    \ap = \preap/\dG.
\end{equation}
\end{defn}
Let also $W_{\cO(\chi)}$\index{Weyl Groups!$W_{\cO(\chi)}$} denote the subgroup of $W_G$\index{Weyl Groups!$W_G$} whose elements satisfy $w\cO(\chi) = \cO(\chi)$. 
In Section \ref{s:NormEisen} we argue that there is a bijection between $\ap$ with a well-understood subset $\Xi  \subset W_{\cO(\chi)}\backslash \fa^{G*}$\index{$\Xi$}. Using this bijection, we can define the Hilbert space $L^2(\Xi,\mu_0)$, for an explicitly defined positive measure $\mu_0$\index{Measures!$\mu_0$} on $\Xi$, smooth on each component of $\Xi$.

\subsection{Statements}
Our goal is to decompose the space $\Ltwoxi_{[T,\cO(\chi)]}^K$ of (\ref{eq:sphspace}) as a module for the global spherical Hecke algebra $\cH(G(\bbA_F),K) = \otimes_v'\cH(G_v,K_v)$\index{$\cH(G(\bbA_F),K)$} and explicitly describe the spectral measure. Here, $\cH(G_v,K_v)$\index{$\cH(G(\bbA_F),K)$!$\cH(G_v,K_v)$} denotes the corresponding spherical Hecke algebra of each local place $v$. 
We will separate this problem into two parts. First, we propose a modification of the Eisentein series $\cE(\nu,\chi,g)$ and define the normalized Eisenstein series $\cE_0(\nu,\chi,g)$ (see Definition \ref{def:normEisen}), which as functions of $\nu\in X_T^G$, they are holomorphic and $W_{\cO(\chi)}$-invariant. Taking wave packets with respect to these normalized Eisenstein series we define the normalized subspace $\Ltwoxi_{[T,\cO(\chi)],0}^K$\index{$\Ltwoxi_{[T,\cO(\chi)]}^K$!$\Ltwoxi_{[T,\cO(\chi)],0}^K$} inside $\Ltwoxi_{[T,\cO(\chi)]}^K$\index{$\Ltwoxi_{[T,\cO(\chi)]}^K$} and afterwards we deal with its orthogonal complement. 

The algebra $\cH(G(\bbA_F),K)$ is equipped with a $\ast$-structure coming from each local factor. It acts by convolution on the automorphic space and diagonally on $\Xi$, due to the fact that the Eisenstein series (and also their normalized version) form a family of eigenfunctions for the Hecke operators. Our first main result is the following:

\begin{thmintro}\label{thm:LowerBound}
The transform $\fF:\Ltwoxi_{[T,\cO(\chi)],0}^K\to L^2(\Xi,\mu_0)$\index{$\fF$}
\begin{equation}\label{eq:FTransf}
\fF(f)(\nu) = \int_{\aq} f(g) \cE_0(-\nu,\overline{\chi},g) dg
\end{equation}
defined for $\nu\in\Xi$ is an isometric isomorphism of $\ast$-unitary $\cH(G(\bbA),K)$-modules.
\end{thmintro}

This lower bound to the unramified spherical automorphic spectrum, Theorem \ref{thm:LowerBound}, has some important consequences. Given a parameter $\psi$, by item (c) of Definition \ref{d:APspace}, we let $\dHc \subseteq \dG$\index{$G$!$\dHc$} be the identity component of the centraliser in ${}^LG$ of the image of $\varphi_\chi$. This is a reductive subgroup containing the subtorus $\dT^\Gamma$ of $\hat{T}$ of Galois invariants. Assuming further that $\psi$ is discrete implies that the group $\dHc$ is a semisimple subgroup of $\dG$ of maximal rank and that $\rho = \rho_\bo$ is a homomorphism $SL(2,\bbC)\to \dHc\subseteq \dG$ corresponding to a distinguished nilpotent orbit $\bo$ of $\Lie(\dHc)$. Let $2\lambda(\bo) \in  \fa^*_{\bbC}\subseteq \dt$ be the corresponding weighted Dynkin diagram of the orbit $\bo$ (see \cite[Chapter 5]{C}). Let also $\nu(\bo)\in X_T$ be the image of $\lambda(\bo)$ under the surjection $\fa_{\bbC}^*\to X_T$. We thus have an automorphic character $\chi_{\nu(\bo)}$ whose local (unramified) factors will be denoted as $\chi_{v,\nu(\bo)}$. For any place $v$, denote by $\pi_{\chi_{v,{\nu(\bo)}}}$ the unique irreducible spherical subquotient of the (unramified) principal series representation of the local group $G_v$ induced by the character $\chi_{v,\nu(\bo)}$ of $A_v$\index{$G$!$A_v$} and denote by $\pi_{\chi_{\nu(\bo)}}$ the irreducible representation of $G(\bbA_F)$ given by
\begin{equation}\label{eq:localfactors}
\pi_{\chi_{\nu(\bo)}} = \otimes'_v\pi_{\chi_{v,{\nu(\bo)}}},
\end{equation}
which has a $ K $-invariant vector.  Each $K$-invariant normalised Eisenstein series $\cE_0(\nu(\bo),\chi,g)$ generates an admissible $G(\bbA_F)$-subrepresentation of the automorphic space $L^2(G(F)\backslash G(\bbA_F),\xi)_{[T,\cO(\chi)]}$ (\cite{HC}, see also the survey \cite[Theorem 3.5]{Co}) whose space of $K$-invariants is one dimensional. From the unitary structure of the automorphic space, it follows that this subrepresentation is irreducible and hence isomorphic to $\pi_{\chi_{\nu(\bo)}}$. We will say that two nilpotent orbits $\bo$ and $\bo'$ of $\Lie(\dHc)$ are equivalent if they are conjugate under the adjoint action of $C_{\dG}(\varphi_\chi)$. The $G(\bbA_F)$-representations $\pi_{\chi_{\nu(\bo)}}$ and $\pi_{\chi_{\nu(\bo')}}$ are then equivalent if and only if $\bo$ and $\bo'$ are equivalent. Now let\index{$\Ltwoxi_{[T,\cO(\chi)]}^K$!$L^2(G(F)\backslash G(\bbA_F),\xi)_{[T,\cO(\chi)],0}^{\textup{disc,sph}}$}
\[
L^2(G(F)\backslash G(\bbA_F),\xi)_{[T,\cO(\chi)],0}^{\textup{disc,sph}}
\]
denote the smallest closed and $G(\bbA_F)$-invariant subspace of the discrete part of $L^2(G(F)\backslash G(\bbA_F),\xi)_{[T,\cO(\chi)],0}$ (i.e., the closure of the span of all topologically irreducible $G(\bbA_F)$-subrepresentations) containing its $K$-invariant vectors. We obtain the following representation-theoretic corollaries:

\begin{corintro}\label{cor:DiscreteSpec}
The space $L^2(G(F)\backslash G(\bbA_F),\xi)_{[T,\cO(\chi)],0}^{\textup{disc,sph}}$ is multiplicity-free and decomposes as
\begin{equation*}
L^2(G(F)\backslash G(\bbA_F),\xi)_{[T,\cO(\chi)],0}^{\textup{disc,sph}} = \bigoplus_{\bo} \pi_{\chi_{\nu(\bo)}} ,
\end{equation*}
with the sum indexed by representatives $\{\bo\}$ of the finite set of equivalence classes of distinguished unipotent orbits of $\dHc\subseteq {\hat G}$.
\end{corintro}

\begin{corintro}\label{cor:Unitary}
For each distinguished orbit $\bo$ of $\dHc\subseteq{\hat G}$, the local factors $\pi_{\chi_{v,\nu(\bo)}}$ of $\pi_{\chi_{\nu(\bo)}}$ are unitary representations for all places $v$.
\end{corintro}
Corollary \ref{cor:Unitary} is in accordance to Arthur's conjectures \cite{A}, and we note that in \cite{Mi}, for $G$ split and $\chi$ trivial, results on the unitarity of the  unique irreducible spherical subquotient of the (unramified) principal series of the local factors with parameter $\nu(\bo)$ were also obtained. 

We expect that when $F$ is a function field the results analogous to Theorem \ref{thm:LowerBound} and their corollaries hold true. For such fields we are currently unaware if the existence of an origin in the sense  of Definition \ref{def:origin} and Proposition \ref{prop:originNF} holds true. In the absence of such origin, the explicit formulas for the spectral measures involved will need to be modified accordingly.

As for the orthogonal complement to the normalized spectrum, suppose that $G$ is split and $\chi$ is the trivial character. In \cite{DMHO}, the existence of a cascade of contour shifts with good properties was established. Using this notion, we can show that even without the normalization of the Eisenstein series there is no additional contribution to the spectrum\footnote{In fact, we prove a more general result in Theorem \ref{t:MainResult2}.}. We thus arrive at the second
main result of this paper:

\begin{thmintro}\label{thm:MainTheorem}
Let $G$ be split. 
The orthocomplement of the normalized spherical subspace is zero:
\begin{equation}\label{eq:norm=all}
\Ltwoxi_{[T,\cO(1)]}^K = \Ltwoxi_{[T,\cO(1)],0}^K.
\end{equation}
\end{thmintro}

For this last statement, which is equivalent to the statement that there is no extra contribution to the spectrum due to zeroes of $L$-function occuring in the denominators of the expression of the inner product between pseudo-Eisenstein series, we used a classical contour-shift approach building on the 
contributions of Langlands \cite{La1}, \cite{La2}, Jacquet \cite{J}, Kim \cite{K1}, \cite{K2},  Moeglin \cite{M1}, \cite{M2}, 
Moeglin-Waldspurger \cite{MW1}, \cite{MW2} and others. Recently, a breakthrough geometric explanation for this surprising behaviour has been given by David Kazhdan and Andrei Okounkov \cite{KO}. 
\section{The inner product formula}\label{s:InnerProd}

\subsection{Eisenstein and pseudo-Eisenstein series}
In order to fix notations, we review the inner product formula for pseudo Eisenstein series (see \cite{La1} and \cite{MW2}). Given $[T , \cO(\chi)]$, we extend the unramified automorphic character $\chi$ of $T$ to the whole Borel subgroup $B$, by setting it to be trivial in the unipotent radical $U$. Let $t_B : G(\bbA_F) \to T(\bbA_F)/(T(\bbA_F) \cap K)$ be the map obtained from the Iwasawa decomposition $G(\bbA_F) = B(\bbA_F)K$. Composing $t_B$ with the adelic norm map yields a function $G(\bbA_F) \to T(\bbA_F)/T^1 \cong X_*(T)^\Gamma\otimes \Val_F$ , denoted $g\mapsto |t_B(g)|$. Since $\chi$ is trivial on $T(\bbA_F) \cap K$, we have that $\chi(t_B(g))$ is well defined. We let the Eisenstein series corresponding to $\chi$ be defined by \index{$\cE(\nu,\chi,g)$}
\begin{equation}\label{eq:Eiseinstein}
\cE(\nu,\chi,g) = \sum_{\gamma \in B(G)\backslash G(F)} \chi(t_B(\gamma g))|t_B(\gamma g)|^{\nu + \varrho},
\end{equation}
for $\nu \in X_T$ , $g \in G(\bbA_F)$ and where $\varrho$ is half the sum of all positive roots. This series is absolutely convergent whenever $\re(\nu - \varrho)$ is in the interior of the dominant Weyl chamber \cite[Section II.1.5]{MW2}.

We denote by $\cP(X_T^G)$\index{$\cP(X_T^G)$} the space of Paley-Wiener functions on $X_T^G$. Recall that when $F$ is a number field, we identify $X_T^G\cong\fa_\bbC^{G*}$, so this is the space of holomorphic functions $\phi$ on $X_T^G$ satisfying the usual Paley-Wiener estimates. Given $\fR>0$, we let $\cP^\fR(X_T^G)\supset \cP(X_T^G)$\index{$\cP(X_T^G)$!$\cP^\fR(X_T^G)$} be the space of all functions $\phi$ which are holomorphic on $B_\fR = \{\lambda \in X_T^G\mid |\re\lambda|<\fR\}$ satisfying the similar estimate that, for all $n\in\bbN$ there is $C_n>0$ with
\begin{equation}\label{eq:PWRest}
|\phi(\lambda)| \leq C_n(1 + \|\lambda\|)^{-n}.
\end{equation}
\noindent One defines the pseudo-Eisenstein series\index{$\cE(\nu,\chi,g)$!$\theta_\phi$}, for $\phi \in \cP(X_T^G)^{\fR}$ (see \cite[Section II.1]{MW2}) via
\begin{equation}\label{eq:pseudo-Eisenstein}
\theta_\phi(g)=\int_{\operatorname{Re}(\nu)=
\nu_0}\phi(\nu) \cE(\nu,\chi,g)d\nu,
\end{equation}
where $\nu_0\in \re X_T^G$ is in the cone of convergence of the Eisenstein series.

\subsection{Inner product formula}

The starting point for our results is the inner product formula between pseudo-Eisenstein series \cite[II.2.1]{MW2}. In our context, the formula will be given by computing the action of the global intertwining operator on the global spherical vector. The computation (of the formula) in the generality considered in this section was done in \cite{H4}. In the computation, we give $K$ the Haar measure whose total measure is one. For the unipotent groups, the measure comes from the measure $\mu$\index{Measures!$\mu$} on $\bbA_F$ for which $\mu(\bbA_F/F) = 1$. For a finite place $v$ of $F$, the local measure is the self-dual one and is such that $\mu_v(\fo_v) = (N\fd_v)^{-1/2}$, where $N\fd_v$ denotes the norm of the different of $F_v$ (see \cite[Section 2.2]{Ta}, also for the measures in the Archimedean places). If, on the other hand, we consider the measure $\mu' = \prod_v\mu'_v$\index{Measures!$\mu'$} in which $\mu'_v(\fo_v) = 1$ for $v$ finite (this measure is the relevant one in the computations in \cite[Section 3]{Ca}, for example) we have $\mu' = \delta_F\mu$, in which
\begin{equation}\label{eq:delta}
\delta_F = |d_F|^{1/2}
\end{equation}
with $d_F$ the discriminant of $F$. 
We similarly define $\delta_{F'}$  for a finite extension $F'$ of $F$. In what follows, given $\eta$ an automorphic character of a split torus, we will write $\eta\sim \eta'$ if $\eta$ and $\eta'$ are unramified twists of each other.

In this case, $X_T^G$ is the complex vector space $\fa^{G*}_{\bbC}$.  Given any finite extension $F'\supseteq F$ and an automorphic character $\eta:C_{F'}\to \bbC^\times$, where $C_{F'}$ denotes the id\`ele class group of $F'$, let $L_{F'}(s,\eta)$\index{L-functions!$L_{F'}(s,\eta)$} be the product over all finite places of the local $L$-functions $L(s,\eta_v)$. For an everywhere unramified character $\eta$, let \index{L-functions!$\Lambda_{F'}(s,\eta)$}
\begin{equation}\label{eq:LambdaNF}
\Lambda_{F'}(s,\eta) := \delta_{F'}^s\Gamma_{F'}(s,\eta)L_{F'}(s,\eta),
\end{equation}
where $\Gamma_{F'}(s,\eta)$ is the product of all gamma-factors at the Archimedean places and $\delta_F$ is as in (\ref{eq:delta}).

\begin{prop}\label{prop:analyticprops}
The functions in (\ref{eq:LambdaNF}) satisfy the following properties:
\begin{enumerate}
\item[(a)] $\Lambda_{F'}(s,\eta)$ is entire, unless $\eta\sim 1$. For $\eta =1$, it has simple poles for $s\in\{0,1\}$.
\item[(b)] $\Lambda_{F'}(1-s,\eta^{-1})=\epsilon(\eta) \Lambda_{F'}(s,\eta)$, for an $\epsilon(\eta)\in\bbC$ such that $\epsilon(1) = 1$.
\item[(c)] $\Lambda_{F'}(s,\eta)$ has zeroes only for $0<\textup{Re}(s)<1$.
\item[(d)] $\Lambda_{F'}(s,\eta)$ is essentially bounded on vertical strips.
\item[(e)] $\Lambda_{F'}(s,\eta)^{-1}$ is of moderate growth on the right hyperplane $\textup{Re}(s)\geq 1$.
\end{enumerate}
\end{prop}

\begin{proof}
The analytic properties (a)-(d) are all essentially done in \cite{Ta} (see also \cite[Chapter VII, \textsection6 and \textsection7]{We} and \cite[Section 3.1]{Bu}). For item (e), see \cite{GL}. 
\end{proof}
The growth properties of items (d) and (e) are important to guarantee the convergence of the integrals we will be dealing with afterwards. We also define a regularised version of the $\Lambda$-functions above:\index{L-functions!$\Lambda^{\textup{reg}}_{F'}(s,\eta)$}
\begin{equation}\label{eq:regularisedLambda-NF}\index{$\Lambda^{\textup{reg}}_{F'}$}
\Lambda^{\textup{reg}}_{F'}(s,\eta):=\left\{
\begin{array}{ll}
-s(s-1)\Lambda_{F'}(s,\eta) & \textup{if }\eta = 1\\
\Lambda_{F'}(s,\eta) & \textup{if }\eta \not\sim 1.
\end{array}
\right.
\end{equation}
These are entire functions on $\bbC$ satisfying the properties (b)-(e) above.

The $F$-Borel $B$ determines a set ${R_+}$ of positive roots in the relative root system ${R = R(G,A)}$. Given a reduced $\alpha\in{R_+}$, we will denote by $G_\alpha$\index{$G$!$G_\alpha$} the corresponding rank one subgroup of $G$ and by $\tilde{G}_{\alpha,D}$\index{$G$!$\tilde{G}_{\alpha,D}$} the simply connected covering of its derived group.
As is well-known, either $\tilde{G}_{\alpha,D}$ is isomorphic to $\Res_{F_\alpha/F}SL_2$ or $\Res_{F_\alpha/F}SU(2,1)$, where $F_\alpha$\index{$\eta_\alpha$!$F_\alpha$} is a suitable finite and separable extension of $F$.
\begin{defn}\label{def:RootTypes}
A reduced root $\alpha$ will be said to be of $SL_2$- or $SU(2,1)$-type if the group $\tilde{G}_{\alpha,D}$ is isomorphic to, respectively, $\Res_{F_\alpha/F}SL_2$ or $\Res_{F_\alpha/F}SU(2,1)$.
\end{defn}
Now let $E_\alpha\subseteq E$\index{$\eta_\alpha$!$E_\alpha$} denote the minimal splitting field of $\tilde{G}_{\alpha,D}$. We have $F_\alpha = E_\alpha$ for $\alpha$ of $SL_2$-type, and $F_\alpha\subseteq E_\alpha$ is the degree two extension that defines $SU(2,1)$ over $F_\alpha$, in the other case. For our purposes, it will be convenient to consider fields and other objects related to every positive relative root, not just the reduced ones. 

\begin{defn}\label{def:autochars}
Given positive relative roots $\alpha\in R_+$, define the field $K_\alpha$ and the autormorphic character $\eta_\alpha$\index{$\eta_\alpha$} of $K_\alpha$\index{$\eta_\alpha$!$K_\alpha$} via
\begin{itemize}
\item For $\alpha$ reduced and of $SL_2$-type: let $K_\alpha := E_\alpha$ and $\eta_\alpha$ be the trivial automorphic character of $K_\alpha$.
\item For $\alpha$ reduced and of $SU(2,1)$-type: let $K_\alpha := E_\alpha$ and $\eta_\alpha$ be the trivial automorphic character of $K_\alpha$.
\item For $\beta = 2\alpha$ non-reduced (so $\alpha$ of $SU(2,1)$-type): let $K_\beta := F_\alpha$ and $\eta_\beta$ be the automorphic character $\eta_{E_\alpha/F_\alpha}$, corresponding to the quadratic extension $F_\alpha\subseteq E_\alpha$.
\end{itemize}
We extend the definitions of $K_\alpha$ and  $\eta_\alpha$ for all $\alpha\in R$ by setting $K_{-\alpha} = K_\alpha$ and $\eta_{-\alpha} = \eta_\alpha$, for every $\alpha\in R_+$. 
\end{defn}

Next, for any finite extension $F\subseteq F'$, let $\Gamma_{F'}$ denote the absolute Galois group of $F'$. For any $\alpha = \res(\tilde\alpha)\in R_+$ (and $\tilde\alpha\in\tilde{R}$ an absolute root) let $\Gamma_{K_\alpha/F}$ be a complete set of representatives of the quotient $\Gamma_F/\Gamma_{K_\alpha}$. Define the $F$-rational map \index{Root Systems!$\underline{\alpha}^\vee$} $\underline{\alpha}^\vee:\Res_{K_\alpha/F}\bbG_m\to T$ by: for all $(x_\sigma)_\sigma\in (\Res_{K_\alpha/F}\bbG_m)(F^{\textup{s}})$, where $F^{\textup{s}}$ denotes the separable closure of $F$,
\begin{equation}\label{eq:under_root}
\underline{\alpha}^\vee((x_\sigma)_\sigma):= \prod_{\sigma\in \Gamma_F/\Gamma_{K_\alpha}}({}^\sigma\tilde\alpha^\vee)(x_\sigma),
\end{equation}
where ${}^\sigma\tilde\alpha^\vee$ denotes the action of $\sigma\in\Gamma_F$ on the absolute coroot $\tilde\alpha^\vee$. For $v$ a place of $F$ we let $\underline{\alpha}_v^\vee$ be $\underline{\alpha}^\vee$ over the $F_v$-rational points, i.e., $\underline{\alpha}_v^\vee : (\Res_{K_\alpha/F}\bbG_m)(F_v)\to T(F_v)$. In particular, $\chi\circ\underline{\alpha}^\vee$ will stand for the automorphic character $\otimes_v(\chi\circ\underline{\alpha}_v^\vee)$.

With the definitions above in place, for any $\alpha\in R_+$ and $\lambda\in X_T = \fa^*_\bbC$, we define the ratio  $\sr_\alpha(\lambda,\chi)$\index{$r(\lambda,\chi)$!$\sr_\alpha(\lambda,\chi)$} by
\begin{equation}\label{eq:r-factorNF}
\mathsf{r}_\alpha(\lambda,\chi) := \displaystyle\frac{\Lambda_{K_\alpha}(\langle\lambda,\res(\tilde\alpha^\vee)\rangle,(\chi\circ\underline{\alpha}^\vee)\eta_\alpha)}{\epsilon((\chi\circ\underline{\alpha}^\vee)\eta_\alpha)\Lambda_{K_\alpha}(\langle\lambda,\res(\tilde\alpha^\vee)\rangle+ 1,(\chi\circ\underline{\alpha}^\vee)\eta_\alpha)}.
\end{equation}
Here, $\langle\cdot,\cdot\rangle$ is the complex bilinear pairing of $\fa^*_\bbC\times\fa_\bbC$. We see 
$\tilde\alpha^\vee\in \bbR\otimes X_*(T)$ as a root of $\dg= \Lie(\dG)$\index{$\dg$} with respect to $\dt=\Lie(\dT)$\index{$\dg$!$\dt$}, and thus a linear functional on $\dt$. Then, $\res(\tilde\alpha^\vee)\in\fa$ 
is its restriction to $\fa^*_\bbC\subseteq\dt$.
\begin{rem}
Suppose that $G = SU(2,1)$ and $F\subseteq E$ is the quadratic extension that defines $G$. Let $\alpha\in R_+$ be the reduced root and $\beta = 2\alpha$. In this case we note that $\underline{\alpha}^\vee = \underline{\beta}^\vee$ and the product $\sr_\alpha(\lambda,\chi)\sr_\beta(\lambda,\chi)$ reduces to
\begin{equation}\label{eq:SU(2,1)-case}
\displaystyle\frac{\Lambda_{E}(\frac{1}{4}\langle\lambda,\alpha^\vee\rangle,\chi\circ\underline{\alpha}^\vee)}{\epsilon(\chi\circ\underline{\alpha}^\vee)\Lambda_{E}(\frac{1}{4}\langle\lambda,\alpha^\vee\rangle+1,\chi\circ\underline{\alpha}^\vee)}
\displaystyle\frac{\Lambda_{F}(\frac{1}{2}\langle\lambda,\alpha^\vee\rangle,(\chi\circ\underline{\alpha}^\vee)\eta_{\beta})}{\epsilon((\chi\circ\underline{\alpha}^\vee)\eta_{\beta})\Lambda_{F}(\frac{1}{2}\langle\lambda,\alpha^\vee\rangle+1,(\chi\circ\underline{\alpha}^\vee)\eta_{\beta})},
\end{equation}
where, compared with (\ref{eq:r-factorNF}), we used here $\alpha^\vee$, the dual of the relative root $\alpha$, instead of $\res(\tilde\alpha^\vee)$.
\end{rem}
Given any $\lambda\in X_T$, denote by $I(\chi_\lambda)$ the representation of $G(\bbA_F)$ induced from the character $\chi_\lambda$ of $B(\bbA_F)$. For any $w\in W_G$, chose a representative $\tilde w\in G(F)$ and let $M(\lambda,\chi, w)$ denote the global intertwining operator
$$(M(\lambda,\chi,w)f)(g) = \int_{U(\bbA_F)/(U(\bbA_F)\cap \tilde wU(\bbA_F)\tilde w^{-1})}f(\tilde{w}^{-1}ug)du,$$
defined for $\re\lambda$ sufficiently positive, and mapping $I(\chi_\lambda)\to I(w\chi_\lambda)$. Let $(\cdot,\cdot)$ denote the Hermitian product on the space of automorphic forms. Given a root system $\Psi$ and $w$ an element in its Weyl group, we will write $\Psi(w) := \Psi_+\cap w^{-1}\Psi^-$. We have:

\begin{prop}[\cite{H4}]\label{prop:Shahidi}
If $1_K$ denotes the normalised spherical vector, we have
$$(1_K, M(\lambda,\chi,w)1_K) = \prod_{\alpha\in R(w)}\sr_{\alpha}(\lambda,\chi),$$
with $\sr_\alpha(\lambda,\chi)$ given as in (\ref{eq:r-factorNF}).
\end{prop}

\begin{proof}
This is the main result of \cite{H4}.
\end{proof}

Recall that $\cO(\chi) = \{\chi_\nu\mid\nu\in X_T^G\}$ denotes the orbit of unramified character twists of $\chi$ and that $W_{\cO(\chi)} = \{w\in W_G\mid w\chi\in\cO(\chi)\}$\index{Weyl Groups!$W_{\cO(\chi)}$}. We have (conf. \cite[II.2.1]{MW2}):

\begin{cor}\label{cor:NF.InProd}
The inner product between two pseudo-Eisenstein series can be written as
\begin{align}\label{eq:inprodpseueisNF}
(\theta_{\phi},\theta _{\psi})=&\int\limits_{\re(\lambda) = \lambda_0}
\sum_{{w}\in W_{\cO(\chi)}}\prod_{\alpha\in R(w)}\sr_{\alpha}(\lambda,\chi)\phi^{-}(-w\lambda)\psi (\lambda)d\lambda,
\end{align}
in which $\lambda_0\in \fa^{G*}$ is a fixed element in the cone of convergence of $\cE(\chi,\lambda,g)$\index{$\cE(\nu,\chi,g)$} and \index{$\phi^-$}
\begin{equation}\label{eq:phiminus}
\phi^{-}(\lambda):=\overline{\phi(\overline{\lambda})}.
\end{equation}
\end{cor}

\begin{rem}\label{rem:measures}
The measure $d\lambda$\index{Measures!$d\lambda$} used in (\ref{eq:pseudo-Eisenstein}) and (\ref{eq:inprodpseueisNF}) is defined as follows. Let $\cX := X^*(T)^\Gamma\cap \fa^{G*}\subseteq \fa^{G*}$\index{$\cX$} and let $\cY\subseteq \fa^G$\index{$\cX$!$\cY$} be the corresponding dual lattice. Note that $\cY\subseteq \fa^G$ is the the orthogonal projection of $X_*(T)^\Gamma$ onto $\fa^G$ and that $X_T^G = \bbC\otimes\cX = \fa_{\bbC}^{G*}$. Now let $\{y_1,\ldots,y_r\}$ be a basis of $\cY$. Then, the measure $d\lambda$ on subspace $\lambda_0+i\fa^{G*}$ of $\fa_{\bbC}^{G*}$ is given by
\begin{equation}\label{eq:measureV}
d\lambda := (-2\pi i)^{-r}dy_1\wedge\cdots\wedge dy_r.
\end{equation}
\end{rem} 

\subsection{Rewriting the scalar product}\label{sub:Rewriting} In this subsection, our aim is to rewrite the expressions for the scalar product  (\ref{eq:inprodpseueisNF})
in a way that we can use the Residue Calculus developed in \cite{HO1} and refined in \cite{O-Sp} and \cite{O-Supp}. We will also use the notion of Kac (co)root system of $\chi$, to be defined below, inspired by \cite{Re2}.

Recall that $\dg$\index{$\dg$} denotes the Lie algebra of $\dG$. Fix $\dt\subseteq\db\subseteq\dg$ a Cartan and a Borel subalgebras corresponding to $\dT$ and $\dB$\index{$G$!$\dB$}. We have a decomposition
\begin{equation}\label{eq:LieDecomp}
\dg  = \dn^-\oplus\dt\oplus\dn,
\end{equation}
with $\dn$\index{$\dg$!$\dn$} the nilpotent radical of $\db$\index{$\dg$!$\db$} and $\dn^-$\index{$\dg$!$\dn^-$} its opposite. 

Recall also that we denote by $\tilde{R}\subseteq X^*(T)$\index{Root Systems!$\tilde{R}$ absolute roots} the absolute root system and by $R\subseteq X^*(A)$\index{Root Systems!$R$ relative roots} the relative root system, which is the root system of the pair $(G,A)$.
We will denote by $\tilde{R}^\vee \subseteq  X_*(T)$\index{Root Systems!$\tilde{R}^\vee$ absolute coroots}, the root system of absolute coroots, which is viewed as the root system of the pair $(\dg,\dt)$. 
From the identification $\dt\cong\bbC\otimes X^*(T)$, we view $\fa^*_\bbC = \bbC\otimes X^*(T)^\Gamma$ as a subspace of $\dt$. 
We can therefore consider $\res(\tilde\alpha^\vee)\in\fa$, the restriction of $\tilde\alpha^\vee$ to $\fa_\bbC^*$. We let $\rR\subseteq\fa$\index{Root Systems!$\rR$ restricted coroots} be the possibly non-reduced root system of restricted coroots.
We note that this equals to the set of roots of the pair ($\dg,\fa^*_\bbC$).

\begin{rem}
We note that, in general, $\rR$ is not the dual root system of $R$, but we still have a bijective correspondence $\alpha\leftrightarrow\ra$
which goes as follows: if $\alpha\in R$, then $\ra:= \res(\tilde\alpha^\vee)$, where $\tilde\alpha\in\tilde{R}$ is an absolute root that restricts to $\alpha$ and $\tilde\alpha^\vee\in\tilde{R}^\vee$ is the (usual) dual coroot of $\tilde\alpha$. Under this correspondence, $\alpha$ reduced implies $\ra$ reduced, and vice-versa.
\end{rem}

\begin{defn}\label{defn:KacRts}
The Kac root system of $\chi$, denoted $\rR_\chi$\index{Root Systems!$\rR_\chi$ Kac coroots}, is defined as the set of all restricted coroots $\ra$ whose corresponding relative root $\alpha\in R$ satisfy
\begin{equation}\label{eq:Kac}
\chi\circ\underline\alpha^\vee = \eta_\alpha,
\end{equation}
where  $\eta_\alpha$ was defined in Definition \ref{def:autochars} and $\underline{\alpha}^\vee$ in (\ref{eq:under_root}). 
\end{defn}

\begin{prop}
The Kac root system $\rR_\chi\subseteq\rR$ can be identified with the root system of the pair $(\dHc,\dT^\Gamma)$, where $\dHc:=C_{{}^LG}(\varphi_\chi)^\circ$\index{$G$!$\dHc$} is the identity component of the centralizer in ${}^LG$ of the Langlands parameter of the unramified character $\chi$.
\end{prop}

\begin{proof}
Following \cite{Re2}, we define an equivalence relation on $\tilde{R}^\vee$ by saying that $\tilde\alpha^\vee$ is $\Gamma$-equivalent to $\tilde\beta^\vee$ if and only if $\res(\tilde\alpha^\vee)$ is positively proportional to $\res(\tilde\beta^\vee)$. Denote the set of equivalence classes by $\rR_\Gamma$. One obtains only two types of equivalence classes $a\in \rR_\Gamma$: if we choose an element $\tilde\beta_a^\vee\in a$ such that $\rb_a = \res(\tilde\beta_a^\vee)$ is reduced, then the corresponding relative root is either of $SL_2$- or of $SU(2,1)$-type (this is, respectively, the type I and type II equivalence classes in \cite[Section 3.3]{Re2}, adapted to our context). Proceeding in a similar fashion as was done in \cite[Section 3.4]{Re2}, one considers the twisted affine extension of the root system $\{\rb_a\mid a\in\rR_\Gamma\}$ and notes that the set of linear parts of all affine roots in this twisted affine system is $\rR$, the set of restricted coroots. Finally, using that $\chi\circ\underline\alpha^\vee$ and $\varphi_\chi\circ\underline{\alpha}^\vee$ correspond to each other by class field theory, one can show that the condition (\ref{eq:Kac}) defines a reduced root system $\rR_\chi\subseteq\rR$, with the claimed property (we refer also to the explicit computations of the root system of fixed-point subalgebras in \cite[Section 3.4]{Re2}). This finishes the proof.
\end{proof}

The group $W_{\cO(\chi)}\subseteq W_G$\index{Weyl Groups!$W_{\cO(\chi)}$} acts on $\rR_\chi$. We let ${}^rW_{\cO(\chi)}$\index{Weyl Groups!${}^rW_{\cO(\chi)}$} denote the subgroup of $W_{\cO(\chi)}$ generated by the reflections $\{s_\alpha\mid\ra\in\rR_\chi\}$. We let also $W_{\cO(\chi)}^+$\index{Weyl Groups!$W_{\cO(\chi)}^+$} denote the subgroup of $W_{\cO(\chi)}$ that stabilizes $V_+$\index{$V$!$V_+$}, the closure of the dominant chamber defined with respect to $\rR_{\chi,+}:=\rR_+\cap\rR_\chi$\index{Root Systems!$\rR_{\chi,+}$} of $\fa^{G*}$. Elements $w\in W_{\cO(\chi)}^+$ are such that $w(\rR_{\chi,+})=\rR_{\chi,+}$. 

\begin{notation}
To ease the notation, we shall omit the subscript $\cO(\chi)$ for the several finite groups related to the Weyl group and write simply $W=W_{\cO(\chi)}$\index{Weyl Groups!$W$}, $W^+=W_{\cO(\chi)}^+$\index{Weyl Groups!$W^+$} and ${}^rW = {}^rW_{\cO(\chi)}$\index{Weyl Groups!${}^rW$}.
\end{notation}

\begin{defn}\label{def:origin}
We say that an element $\chi_0\in\cO(\chi)$ is an origin for $\cO(\chi)$, if $w\chi_0=\chi_0$ for all $w\in W$.
\end{defn}

\begin{prop}\label{prop:originNF}
If $F$ is a number a field, then there exists an origin $\chi_0$ in $\cO(\chi)$. 
\end{prop}

\begin{proof}
Choose $\chi$ in $\cO(\chi)$. Note that $\cO(\chi)$ is an $X_T^G$-torsor with compatible $W$-action, that is, for all $w\in W$ we have $w\chi=\chi\otimes\nu_w$ for a unique $\nu_w$ in $X_T^G$. Hence, the existence of an origin is equivalent to the vanishing $H^1(W,X_T^G)=0$, but since $X_T^G \cong \fa^{G*}_\bbC$ we have that $X_T^G$ is uniquely divisible, which implies the claim. 
\end{proof}

\begin{rem}
Since $\cO(\chi)$ has an origin as we are working over a number field, from now on, we can and will assume that the character $\chi$ is an origin for $\cO(\chi)$. Hence, whenever $\alpha\in R$ is such that $\chi\circ\underline{\alpha}^\vee \sim \eta_\alpha$ (recall that this means that the quotient between these characters is unramified), then $\chi\circ\underline{\alpha}^\vee = \eta_\alpha$, that is $\ra \in \rR_\chi$. That said, because of the properties of the $\Lambda$-functions (see Proposition \ref{prop:analyticprops}(a)), it follows that the numerator of the ratio-factor $\sr_\alpha(\lambda,\chi)$, defined in (\ref{eq:r-factorNF}), will have poles exactly when $\ra\in \rR_\chi$. Moreover, in the case in which both $\alpha$ and $2\alpha$ are in $R_+$, only one among $\sr_\alpha(\lambda,\chi)$ and $\sr_{2\alpha}(\lambda,\chi)$ will have a singular numerator, as one can see from (\ref{eq:SU(2,1)-case}).
\end{rem}

\begin{prop}\label{prop:V.2}
We have $W = W^+\rtimes {}^rW$. Furthermore,
The group ${}^rW$ is the Weyl group of $\rR_\chi$ and it is such that $w\chi = \chi$. The set $\rR_{\chi,+}$ is a set of positive roots of $\rR_\chi$.
\end{prop}

\begin{proof}
The necessary arguments are similar to the ones given in \cite[Propositions 1.3, 1.12]{H3} and \cite[Proposition 4.2]{H1} (see also \cite[Proposition 5.4]{H2} and \cite{Si}).
\end{proof}

\subsubsection{The number field case} We now start to rewrite Corollary \ref{cor:NF.InProd} in a more convenient way.
\begin{defn}\label{defn:Hecke-cNF}
Define the local $c$-function by \index{$c(\lambda)$}
\begin{equation}\label{eq:c-function}
c(\lambda):=\prod_{\ra\in\rR_{\chi,+}}\frac{\langle\lambda,\ra\rangle + 1}{\langle\lambda,\ra\rangle}.
\end{equation}
for any $\lambda$ in $V_\bbC$.
\end{defn}

Note that the rational function $c$ on $\fa^{G*}_\bbC$ is in fact the $c$-function of the graded Hecke algebra with (infinitesimal) Hecke parameter $k_{\ra}=1$, for all $\ra$ (see \cite[(1.8)]{HO1}). Note also that $c$ is invariant with respect to elements $w \in W_{\cO(\chi)}^+$.

\begin{defn}\label{defn:r-funcNF}
For $s\in\bbC$ and $\alpha\in{R}$, let $\rho_\alpha(s,\chi):= \Lambda^{\textup{reg}}_{K_\alpha}(s,\eta_\alpha(\chi\circ\underline{\alpha}^\vee))$\index{L-functions!$\rho_\alpha(s,\chi)$}, where the regularised version of the $\Lambda$-function was defined in (\ref{eq:regularisedLambda-NF}). For $\lambda\in \fa^{G*}_\bbC$, define the $r$-function by \index{$r(\lambda,\chi)$}
\begin{equation}\label{eq:r-functionNF}
r(\lambda,\chi):=\prod_{\alpha\in{R}_+}\rho_\alpha(\langle\lambda,\ra\rangle,\chi).
\end{equation}
\end{defn}

\begin{lem}\label{lem:EasyComp}
Let $\varphi$ be a function of one complex variable, and define
a function $f$ on $\fa^{G*}_\bbC$ by
$f(\lambda)=\prod_{\beta\in\Psi^+}\varphi(\langle\lambda,\beta\rangle)$, where $\Psi\subseteq \fa^{G}$ is any root system. For any $w$ in the Weyl group of $\Psi$ we have
\begin{equation*}
\frac{f(w\lambda)}{f(\lambda)}=\prod_{\beta\in\Psi(w)}\frac{\varphi(-\langle\lambda,\beta\rangle)}{\varphi(\langle\lambda,\beta\rangle)}.
\end{equation*}
\end{lem}

\begin{proof}
The case when $w$ is simple reflection is straightforward. The result follows by induction on the length of $w$.
\end{proof}

\begin{cor}\label{cor:Easy}
For any $w\in W_{\cO(\chi)}$ and $\lambda \in \fa_\bbC^{G*}$, we have
\begin{equation}\label{eq:r-NF}
\frac{r(\lambda,\chi)}{r(w\lambda,w\chi)}=
\prod_{\alpha\in R(w)}
\frac{\Lambda^{\textup{reg}}_{K_\alpha}(\langle\lambda,\hat{\alpha}\rangle,\eta_\alpha(\chi\circ\underline{\alpha}^\vee))}{\epsilon(\eta_\alpha(\chi\circ\underline{\alpha}^\vee))\Lambda^{\textup{reg}}_{K_\alpha}(\langle\lambda,\hat{\alpha}\rangle+1,\eta_\alpha(\chi\circ\underline{\alpha}^\vee))}.
\end{equation}
and
\begin{equation}\label{eq:zetaNF}
\displaystyle\frac{c(-w\lambda)}{c(-\lambda)}
\displaystyle\frac{r(\lambda,\chi)}{r(w\lambda,w\chi)}=\displaystyle\prod_{\alpha\in R(w)}
\sr_\alpha(\langle\lambda,\ra\rangle,\chi)
\end{equation}
\end{cor}

\begin{proof}
It follows from Lemma \ref{lem:EasyComp} that, for any $w\in W_{\cO(\chi)}$, we have
\begin{equation*}
\frac{r(w\lambda,w\chi)}{r(\lambda,\chi)}=\prod_{\alpha\in{R}(w)}\frac{\rho_{-\alpha}(-\langle\lambda,\ra\rangle,\chi)}{\rho_\alpha(\langle\lambda,\ra\rangle,\chi)},
\end{equation*}
where $\rho_\alpha(s,\chi)$ was defined in Definition \ref{defn:r-funcNF}. The first claim follows by applying the functional equation. Equation (\ref{eq:zetaNF}) follows from (\ref{eq:r-NF}), if we use Lemma \ref{lem:EasyComp} to write the the regularized factors of the the $\Lambda$-function.
\end{proof}

It follows from Corollary \ref{cor:Easy} that the meromorphic function on $\fa^{G*}_\bbC$ defined by (\ref{eq:r-NF}) satisfy the following properties:
\begin{enumerate}[(a)]
\item its pole set is a union of hyperplanes of the form
$\langle\lambda,\hat{\alpha}\rangle=s$, with $-1<\textup{Re}(s)<0$ and
$\alpha\in R(w)$,
\item it is holomorphic and of moderate growth on the union of the right hyperplanes $\langle\lambda,\ra\rangle\geq 0$, with $\alpha\in R(w)$.
\end{enumerate}

The last property follows from properties (d) and (e) recalled in the Proposition \ref{prop:analyticprops}. Inserting (\ref{eq:zetaNF}) in (\ref{eq:inprodpseueisNF})  we obtain that
\begin{equation}\label{eq:inner1}
(\theta_\phi,\theta_\psi)=\int\limits_{\lambda\in \lambda_0 + i\fa^{G*}}
\Sigma_W(\phi^-)(\lambda)\psi(\lambda)
\frac{d\lambda}{c(-\lambda)},
\end{equation}
in which the meromorphic function $\Sigma_W(\phi^-)$\index{$\Sigma_W(\phi^-)$} is given by
\begin{equation}\label{eq:R-NF}
\Sigma_W(\phi^-)(\lambda):=\sum_{w\in W_{\cO(\chi)}}
c(-w\lambda)\phi^-(-w\lambda)\frac{r(\lambda,\chi)}{r(w\lambda,w\chi)}.
\end{equation}
Write $V := \fa^{G*}$\index{$V$} and $V_\bbC$ for its complexification. For any point $p_V\in V$ outside the set of poles of $c(-\lambda)^{-1}$ we define a linear functional $X_{V,p_V}$ on $\cP(V_\bbC)$ by (cf. \cite[equation(3.8)]{HO1}):\index{$X_{V,\lambda_0}$}
\[
X_{V,\lambda_0}(\phi) = \int\limits_{\lambda\in p_V + iV}\phi(\lambda)\frac{d\lambda}{c(-\lambda)},
\]
for all $\phi\in\cP(V_\bbC)$. Hence, setting $p_V=\lambda_0$, we rewrite the expression (\ref{eq:inner1}) as
\begin{equation}\label{eq:inner2}
(\theta_\phi,\theta_\psi)=X_{V,\lambda_0}(\psi\Sigma_W(\phi^-)).
\end{equation}

\section{Normalized Eisenstein series and a lower bound to the spectrum}\label{s:NormEisen}

Equation (\ref{eq:inner2}) would be written in a way to use the Residual Calculus mentioned in the previous section if it were not for the possible singularities coming from zeroes of the $L$-functions in the denominator of (\ref{eq:R-NF}). A first step into overcoming this difficulty is to introduce a normalization to the Eisenstein and pseudo-Eisenstein series.

\subsection{Normalized Eisenstein series}
The Eisenstein series $\cE(\nu,\chi,g)$\index{$\cE(\nu,\chi,g)$} was recalled in (\ref{eq:Eiseinstein}). It is known that this series has meromorphic continuation for all $\nu \in X_T^G$ and satisfies the functional equations
\begin{equation}\label{eq:FEq}
\cE(w\nu,w\chi,g) = \frac{c(w\nu)}{c(\nu)}\frac{r(w\nu,w\chi)}{r(\nu,\chi)}\cE(\nu,\chi,g)
\end{equation}
for all $w\in W$.

\begin{defn}\label{def:normEisen}
For $g\in G(\bbA_F)$ and $\nu \in X_T^G$, we define the normalised Eisenstein series \index{$\cE(\nu,\chi,g)$!$\cE_0(\nu,\chi,g)$}by
\begin{equation}\label{eq:normEisen}
\cE_0(\nu,\chi,g) = \frac{c(-\nu)r(-\nu,\overline\chi)}{|W|}\cE(\nu,\chi,g).
\end{equation}
\end{defn}

\begin{lem}\label{lem:NormEis}
The normalised Eisenstein $\cE_0(\nu,\chi,g)$ is, as a function of $\nu$, holomorphic and $W$-invariant.
\end{lem}
\begin{proof}
Using  (\ref{eq:FEq}) and that $r(\nu,\chi)r(-\nu,\overline{\chi})$ is $W$-invariant,
it follows that $\cE_0(\nu,\chi,g)$ is $W$-invariant.
From (\ref{eq:FEq}), we have that $\cE(\nu,\chi,g)$ has zeroes along the hyperplanes $\langle\hat{\alpha},\lambda\rangle = 0$ for $\alpha\in R$ simple. These zeroes cancel the poles of $c(-\nu)$, and hence it follows that $\cE_0(\nu,\chi,g)$ is holomorphic in the closure of the fundamental chamber. Invoking the $W$-invariance, it is holomorphic everywhere.
\end{proof}

\subsection{Lower bound to the spectrum}
In this subsection, we will use the notations \index{$r(\lambda,\chi)$!$r(\lambda)$}\index{$r(\lambda,\chi)$!$\check{r}(\lambda)$}
\begin{equation}\label{eq:abbrev}
r(\lambda) := r(\lambda,\chi),\qquad\check{r}(\lambda) := r(-\lambda,\overline{\chi}).
\end{equation}

\begin{defn}\label{def:normPseudoEisen}
Given $\phi\in \cP(X_T^G)$, define, for all $g\in G(\bbA_F)$, the normalized pseudo-Eisenstein series via \index{$\cE(\nu,\chi,g)$!$\theta_{\phi,0}$}
\begin{equation}
\theta_{\phi,0}(g) = X_{V,\lambda_0}(\phi\cE_0(\cdot,\chi,g)).
\end{equation}
\end{defn}

\begin{prop}
For any $\phi\in \cP(X_T^G)$ the normalized pseudo-Eisenstein series $\theta_{\phi,0}$ is square-integrable. 
\end{prop}

\begin{proof}
Given $g\in G(\bbA_F)$ and $\phi\in \cP(X_T^G)$, note that
\begin{align*}
\theta_{\phi,0}(g) &= X_{V,\lambda_0}(\phi\cE_0(\cdot,\chi,g))\\
&=\int_{\re(\lambda)=\lambda_0} \phi(\lambda)\cE_0(\lambda,\chi,g)\frac{d\lambda}{c(-\lambda)}\\
&=\frac{1}{|W|} \int_{\re(\lambda)=\lambda_0} r(-\lambda,\overline{\chi})\phi(\lambda)\cE(\lambda,\chi,g)d\lambda,
\end{align*}
that is, $\theta_{\phi,0} = |W|^{-1}\theta_{\check{r}\phi}$, with $\check{r}$ as in (\ref{eq:abbrev}). From the known growth conditions of the $r$-function on vertical strips (see Proposition \ref{prop:analyticprops} and Definition \ref{defn:r-funcNF}), it follows that the function $\check{r}\phi$ is in $\cP^\fR(X_T^G)$ for any $\fR > 0$ and the result follows from \cite[Proposition II.1.2]{MW2}.
\end{proof}
Hence, we are justified to define $\Ltwoxi_{[T,\cO(\chi)],0}^K$ as the subspace of $\Ltwoxi_{[T,\cO(\chi)]}^K$ topologically generated by all the  normalized pseudo-Eisenstein series. Since $\theta_{\phi,0} = |W|^{-1}\theta_{\check{r}\phi}$ for any Paley-Wiener function $\phi$, in view of (\ref{eq:inner2}) we have
\begin{equation}\label{eq:innernorm}
|W|^2(\theta_{\phi,0},\theta_{\psi,0})= X_{V,\lambda_0}((\check{r}\psi)\Sigma_W((\check{r}\phi)^-)).
\end{equation}
Since $r(\lambda,\chi)r(-\lambda,\overline\chi)$ is $W$-invariant and $\check{r}^-(\lambda) = r(-\lambda,\chi)$, from (\ref{eq:R-NF}) we have that
\begin{equation}\label{eq:Rrfunc}
\Sigma_W((\check{r}\phi)^-)(\lambda) = |W|r(\lambda,\chi)A_0(\phi^-)(-\lambda).
\end{equation}
Hence, unlike $(\theta_{\phi},\theta_{\psi})$, the expression for the inner product between normalized pseudo-Eisenstein series has no potential poles corresponding to criticial zeroes of $L$-functions and this expression can be treated directly with Residue Calculus. We now recall the necessary notions and write them in our context.

\subsection{Residue Distributions}
The main references to this part are \cite{HO1}, \cite{O-Sp} and \cite{O-Supp}. Recall that $V=\fa^{G*}$\index{$V$} and $\dim V = r$. Given an affine subspace $L\subseteq V$\index{$\cL$ Residual spaces!$L$}, write $L = c_L +V^L$, where $V^L \subseteq V$\index{$\cL$ Residual spaces!$V^L$} is a (linear) subspace, and $c_L$\index{$\cL$ Residual spaces!$c_L$} is the element of $L$ which is of minimum distance to the origin. In other words, if $V_L$\index{$\cL$ Residual spaces!$V_L$} denotes the orthogonal complement of $V^L$, then $\{c_L\} = V_L \cap L$. We call $c_L$ the center of $L$, and we write $L_\bbC = c_L + V_\bbC^L \subseteq V_\bbC$ for its complexification. We define the tempered form of $L$ to be $L^{\textup{temp}}:=c_L+iV^L \subseteq V_\bbC$\index{$\cL$ Residual spaces!$L^{\textup{temp}}$}. 

The singularities of the $(r,0)$-form $\omega(\lambda) = d\lambda/c(-\lambda)$\index{$\omega$} of the linear functional $X_{V,\lambda_0}$ define a finite collection of affine hyperplanes and we denote by $\cL^\omega$ the lattice of intersection of such hyperplanes. Let also $\cC^\omega$ denote the set of centers of all the elements of $\cL^\omega$. Using the Residue Lemma \cite[Lemma 3.1]{HO1}, it follows that there is a unique collection of tempered distributions $\{X_c | c\in \cC^\omega\}$\index{$X_{V,\lambda_0}$!$X_c$} such that
\begin{itemize}
    \item[(a)] $\supp(X_c)\subseteq \cup \{c + iV^L \mid L\in\cL^\omega\textup{ with }c_L =c\}$
    \item[(b)] $X_{V,\lambda_0} (f) = \sum_{c\in \cC^\omega} X_c (f|_{c + iv^L})$, for all $f \in \cP(V_\bbC)$.
\end{itemize}
Elements in the collection $\{X_c | c\in \cC^\omega\}$ are called local contributions and among them, many are identically zero. In order to obtain a more precise description of the support of $X_{V,\lambda_0}$, it is convenient to consider the symmetrized expression
\[
\frac{1}{|W|}\sum_{w\in W}\frac{1}{c(w\lambda)} = \frac{1}{c(-\lambda)c(\lambda)}
\]
and the linear functional $Y_{V,\lambda_0} \in \cP(V_\bbC)^*$\index{$X_{V,\lambda_0}$!$Y_{V,\lambda_0}$}
\[
Y_{V,\lambda_0}(\phi) = \int\limits_{\lambda\in \lambda_0 + iV}\phi(\lambda)\frac{d\lambda}{c(-\lambda)c(\lambda)},
\]
defined with respect to the $(r,0)$-form of integration $\Omega_r(\lambda) = d\lambda/c(-\lambda)c(\lambda)$\index{$\omega$!$\Omega_r$}. Note that both linear functionals $X_{V,\lambda_0}$ and $Y_{V,\lambda_0}$ can be extended to the spaces $\cP^\fR(V_\bbC)$ for $\fR>0$.

\begin{defn}
Given an affine space $L\subseteq V$ let $\hat{P}(L) = \{\ra\in\rR_\chi\mid\langle\ra,L\rangle = 1\}$ and $\hat{Z}(L) = \{\ra\in\rR_\chi\mid\langle\ra,L\rangle = 0\}$. An affine subspace $L\subseteq V$ is said to be residual if $L = \cap_{\ra\in\hat{P}(L)}\{\lambda\in V\mid \langle\lambda,\ra\rangle = 1\}$ and
\[
|\hat{P}(L)| - |\hat{Z}(L)| - \codim(L) \geq 0.
\]
We shall denote by $\cL$\index{$\cL$ Residual spaces} the collection of residual affine subspaces of $V$ and $\cC$\index{$\cL$ Residual spaces!$\cC$} shall denote the set of centers of elements in $\cL$.
\end{defn}
The next result summarizes the relevance of the residual spaces in decomposing the $Y$-functional. To state this result, a few extra notations are in order. First, recall that $V_+$ denote the closure of the dominant chamber of $V$ with respect to $\rR_{\chi,+}$. Given any $L\in \cL$, let $\rR_{\chi,L}$\index{Root Systems!$\rR_{\chi,L}$} denote the parabolic root subsystem of $\rR_{\chi}$ consisting of all roots that are constant along $L$. From the decomposition $V = V_L\oplus V^L$, we let $\lambda_L$ denote the projection of the base point $\lambda_0$ onto $V_L$ and we let $Y_{V_L,\lambda_L}$\index{$X_{V,\lambda_0}$!$Y_{V_L,\lambda_L}$} denote the corresponding lower rank integral defined on $\cP(V_{L,\bbC})$. Furthermore, if $r^L := \dim V^L$ we let $\{y_1,\ldots,y_{r^L}\}$ be a basis of the orthogonal projection of the lattice $\cY\cap(V_L)^{\perp}$ onto $V^*$ (recall the lattice $\cY$\index{$\cX$!$\cY$} was described in Remark \ref{rem:measures}). 

\begin{thm}\label{thm:Ydecomp}
The following statements hold true:
\begin{itemize}
\item[(1)] There is a unique collection $\{Y_c\mid c\in\cC\}$\index{$X_{V,\lambda_0}$!$Y_c$} of tempered distributions such that,
for any $\phi\in \cP(V_\bbC)$ we have
\[
Y_{V,\lambda_0} (\phi) = \sum_{c\in \cC} Y_c (\phi|_{c + iV^L})
\]
and for all $c\in \cC$, $\supp(Y_c)\subseteq \cup_{\{L\in\cL\mid c_L=c\}} L^{\textup{temp}}$.
\item[(2)] An affine subspace $L$ is residual if and only if $|\hat{P}(L)| - |\hat{Z}(L)| - \codim(L) = 0$.
\item[(3)] For $c\in\cC\cap V_+$ the local contribution $Y_c$ admits a decomposition
\begin{equation}\label{eq:locY}
Y_c(\phi|_{c + iV})=\sum_{\{L\in \cL\;|\;c_L=c\}}Y_{L}({\phi}|_{c + iV}),
\end{equation}
where the functional $Y_{L}$\index{Measures!$Y_L$} is a positive measure on $c_L+iV$ with support on $c_L+iV^L$. 
More precisely, 
\begin{equation*}\label{eq:YL}
Y_{L}({\psi}) := Y_{V_L,c_L}(\{c_L\})\int_{c_L+iV^L}{\psi}(\lambda^L)d\mu^L(\lambda^L),
\end{equation*}
for any test function $\psi$ on $c_L+iV$, in which, $Y_{V_L,c_L}(\{c_L\})$ denotes the total mass of the local contribution $Y_{V_L,c_L}$ of the lower rank integral $Y_{V_L,\lambda_L}$ and for $\lambda^L\in c_L+iV^L$,\index{Measures!$d\mu^L$}
\begin{equation}\label{eq:muL}
d\mu^L(\lambda^L):=\prod_{\ra\in{\rR_{\chi,+}}\backslash\rR_{\chi,L}}
\frac{\langle\ra,c_L\rangle^2+\langle\ra,\im\lambda^L\rangle^2}{(\langle\ra,c_L\rangle-1)^2+\langle\ra,\im\lambda^L\rangle^2}
d\lambda^L,
\end{equation}
with $d\lambda^L$\index{Measures!$d\lambda^L$} a measure on $iV^L$ given by
\begin{equation}
    d\lambda^L = \frac{dy_1\wedge \cdots \wedge dy_{r^L}}{(-2\pi i)^{r^L}}.
\end{equation}
\end{itemize}
\end{thm}

\begin{proof}
Application of the Residue Lemma \cite[Lemma 3.1]{HO1} to $Y_{V,\lambda_0}$ would, a priori, yield a decomposition with respect to the set $\cC^{\Omega_r}$ of centers of affine spaces in the lattice of intersection $\cL^{\Omega_r}$, defined with respect to the singular locus of the form $\Omega_r = d\lambda/c(-\lambda)c(\lambda)$\index{$\omega$!$\Omega_r$}, but from \cite[Theorem 3.13]{HO1} items (1) and (3) follow. Item (2) is \cite[Theorem 3.9]{HO1} (see also \cite[Theorem 1.1]{O-Supp} for a conceptual argument).
\end{proof}

Recall that two nilpotent orbits $\bo,\bo'$ of $\dHc=C_{\dG}(\varphi_\chi)^\circ$\index{$G$!$\dHc$} are said to be equivalent if they are conjugate under the adjoint action of $C_{\dG}(\varphi_\chi)$. This is the case if and only if their corresponding weighted Dynkin digrams are in the same $W$-orbit.

\begin{thm}\label{thm:BalaCarter}
There is a canonical bijection between the set $\{Wc\mid c\in\cC\}$ of $W$-orbits of centers of residual spaces and the set $\{\bo\}$ of equivalence classes of nilpotent orbits of $\Lie(\dHc)$.
\end{thm}

\begin{proof}
Given a nilpotent orbit $\bo$ of $\Lie(\dHc)$, let $2\lambda(\bo)$ denote its weighted Dynkin diagram (WDD) as in the Bala-Carter classification of nilpotent orbits \cite{BC,C}.
There is a canonical bijection between the sets
\[
\left\{\begin{array}{c}\textup{nilpotent orbits } \bo\\ \textup{of }\Lie(\dHc)\end{array}\right\}\leftrightarrow
\left\{\begin{array}{c}\textup{WDD }2\lambda(\bo)\textup{ of nilpotent }\\ 
\textup{orbits of }\Lie(\dHc) \end{array}\right\}\leftrightarrow
\left\{\begin{array}{c}{}^rW\textup{-orbits }{}^rWc\\ \textup{of centers }c\in\cC\end{array}\right\}.
\]
The first bijection is the celebrated result of Bala and Carter while the second bijection is \cite[Introduction, Remark 7.3, Remark 7.5]{O-Sp}, and realized by letting $\lambda(\bo) = c_+$ be the unique dominant element in the ${}^rW$-orbit of $c$. Taking into consideration the action of the complement $W^+$ to ${}^rW$ in $W$ yields the result.
\end{proof}

\begin{prop}
With notations as in Theorem \ref{thm:Ydecomp}, let $c \in \cC\cap V_+$ be dominant and $L\in \cL$ such that $c_L=c$. Then, $Y_{V_L,c_L}(\{c_L\}) > 0$.
\end{prop}

\begin{proof}
In this proof, we shall need information on the representation theory of Hecke algebras. 
The triple $(\dHc,\dB\cap\dHc,\dT^\Gamma)$\index{$G$!$\dB$} determine a based root datum $\cR$ and let $q_0>1$ be the residue field of some completion of $F$ at a non-archimedean place. Denote by $\cH=\cH(\cR,q_0)$ the affine Hecke algebra associated to the pair $(\cR,q_0)$, as in, for example, \cite{O-Sp} and let $\bbH(\cR,1)$ denote the associated graded affine Hecke algebra \cite{Lu} with equal parameter $1$. The Hecke algebra $\cH(\cR,q_0)$ is generated by elements $T_{s}$ labeled by simple reflections $s$ associated to the affine Weyl group of the root system $\cR$. There is an isomorphism $\iota:\cH(\cR,q_0^{-1})\to \cH(\cR,q_0)$ that maps $\iota:T_s \mapsto -q_0^{-1}T_s$ (see \cite[(3.8)]{HO2}), which induces an isomorphism $\bbH(\cR,-1)\to \bbH(\cR,1)$. Furthermore, the tuple $(L,\rR_{L,\chi},\rR_{L,\chi,+})$ naturally determines a parabolic based root datum $\cR_L$ (see \cite[Section 2.2]{O-Sp}). We let ${}^rW_L$ denote the Weyl group of $\rR_{L,\chi}$. 

Now, from the classification of discrete series representations for graded affine Hecke algebras with equal parameters (see \cite{KL,Lu}), we know that
the unique spherical representation with central character $({}^rW_L)c_L$ of the lower rank graded affine Hecke algebra $\bbH_L(\cR_L,-1)$ is a discrete series representation. Furthermore, the unique anti-spherical representation with central character $({}^rW_L)c_L$ is known to be the generic element in the corresponding $L$-packet. Applying the isomorphism $\bbH(\cR,-1)\to \bbH(\cR,1)$, the result follows.
\end{proof}

We finish this section by giving precise algebraic relations between the local contributions of the $X$- and $Y$-distributions.

\begin{defn}\label{def:Averagings}
Given any element $c\in V$, let $W_c$ and ${}^rW_c$ denote the isotropy subgroup of $c$ inside, respectively, $W$ and ${}^rW$. For any function $f$ on $V_\bbC$ define the averaging operator $A_c$ \index{$A_c$} by
\begin{equation}\label{eq:Averagings}
    A_c(f)(\lambda) = \frac{1}{|W_c|}\sum_{w\in W_c}c(w\lambda)f(w\lambda).
\end{equation}
We similarly define $\rA_c(f)$ by summing over ${}^rW_c$ instead of $W_c$.
\end{defn}

\begin{rem}
When $c=0$ is the origin of $V$, $A_0$ is the averaging over the whole Weyl group $W$. Note also that the normalized Eisenstein series satisfies
\[
\cE_0(\lambda,\chi,g) = A_0(r(\cdot,\overline{\chi})\cE(-\cdot,\chi,g))(-\lambda),
\]
for all $g\in G(\bbA)$, so this equation is another way to seeing the $W$-invariance of the normalized Eisenstein series. It also explains the factor $|W|^{-1}$ in (\ref{eq:normEisen}).
\end{rem}

\begin{prop}\label{prop:vanishA}
Let $c\in\cC\cap V_+$. Then for all $w\in  W$ we have:
\begin{equation}\label{eq:tXY}
X_{wc}=Y_{c}\circ w^{*}\circ {A_{wc}}.
\end{equation}
\end{prop}

\begin{proof}
When considering $\rA_{wc}$ instead of $A_{wc}$, this is \cite[Proposition 3.6]{HO1}. For this slightly more general situation, the geometric proof given in the above mentioned proposition of \cite{HO1} still applies.
\end{proof}

\subsection{Symmetrization of the expression for the inner product}
For each $L\in \cL$, we let $\nu_L$\index{Measures!$\nu_L$} be the unique positive measure supported on the tempered form
$L^{\textup{temp}}=c_L+iV^L\subseteq V_\bbC$ characterised by the requirements that, for
all ${f}\in \cP(V_\bbC)$, if $c_L\in V_+$ is dominant, we have
\begin{equation}\label{eq:pushfwd}
\int_{L^{\textup{temp}}}{f}(\lambda)d\nu_L=Y_L(f|_{c_L+iV})
\end{equation}
and \index{Measures!$\nu_{WL}$}
\begin{equation}\label{eq:nuLdef}
\nu_{WL}:=\sum_{L'\subseteq WL} \nu_{L'}
\end{equation}
is a $W$-invariant measure. Moreover, the measure  $\nu_L$ is the push forward of a smooth measure on $c_L+iV^L$
and is of the following form (see \cite[Definition 3.17]{HO1} and Theorem \ref{thm:Ydecomp}(2)): let ${}^rW_L\subseteq {}^rW$ be the Weyl group of the parabolic root subsystem $\rR_{\chi,L}$ and let $|A_{\bo_{{}^rWL}}|$ be the cardinality of the component group of the centralizer of the image of the $\mathfrak{sl}_2$-homomorphism associated to ${}^rWL$ (see Theorem \ref{thm:BalaCarter}). Then, 
for $\lambda = c_L + i\lambda^L\in L^{\textup{temp}}$,
\begin{equation}\label{eq:nuL}
d\nu_L(\lambda)=\frac{|({}^rW_L)_{c_L}|}{|A_{\bo_{{}^rWL}}|}\frac{\prod'_{\ra\in\rR_{\chi,L}}|\langle\ra,c_L\rangle|}{\prod'_{\ra\in\rR_{\chi,L}}|(\langle\ra,c_L\rangle + 1)|}\prod_{ \rR_{\chi,+}\backslash\rR_{\chi,L}}
\frac{\langle\ra,c_L\rangle^2+\langle\ra,\lambda^L\rangle^2}{(\langle\ra,c_L\rangle-1)^2+\langle\ra,\lambda^L\rangle^2}
d\lambda^L,
\end{equation}
where $\prod'$ denotes the product over the nonzero factors. 

\begin{rem}
We remark that the
precise constants in (\ref{eq:nuL}) were not yet available in \cite{HO1}, but these can be derived by a limit procedure
from the explicit formal degree formulas for discrete series characters of Iwahori-Hecke algebras \cite[(0.3)]{Re1} and \cite[Section 4.6]{O-STM} (see also \cite{CKK,CO}). 
These explicit formulas constitute a special case of the proof of the
conjecture \cite[Section 3.4]{HII} for unipotent representations of semisimple groups of adjoint type over 
a non-Archimedean local field.
The limit procedure and the computation of the relevant constants in the present case of graded Hecke
algebras will appear elsewhere \cite{DMO}.
\end{rem}

\begin{lem}\label{lem:easyRewriting}
Let $f\in \cP(X_T^G)$.
We can rewrite the action of the functional $X_{V,\lambda_0}$ as
\[X_{V,\lambda_0}(f) = \sum_{L\in\cL} \int_{L^{\textup{temp}}}{A_0}(f)(\lambda)d\nu_L.\]
\end{lem}

\begin{proof}
For this Lemma, we will use (\ref{eq:tXY}) and a computation similar to \cite[Theorem 3.18]{HO1}, which we reproduce here for convenience (and because we use the slightly different operators ${A_0}$). But before that, recall that $V_+$ is a fundamental domain for the action of ${}^rW$. The group $W^+$ acts on the set $\cC\cap V_+$ of dominant centres of residual subspaces. Now let ${\tilde \cC}^+$ be the set of $W^+$-orbits of elements in $\cC\cap V_+$. It follows that we can write $\cL$ as the union of $W$-orbits of subspaces $L\in\cL$ for which $c_L\in{\tilde \cC}^+$. We then have
\begin{align*}
X_{V,\lambda_0}(f) &= \sum_{c\in \cC} X_c(f|_{c+iV}))\\
&= \sum_{c\in {\tilde \cC}^+} \sum_{w\in W/ W_c}Y_c\circ w^*\left({A_{wc}}(f)|_{wc+iV}\right)\\
&= \sum_{c\in {\tilde \cC}^+} \frac{|W|}{| W_c|}Y_c\left({A_0}(f)|_{c+iV}\right),
\end{align*}
from which the result follows, in view of (\ref{eq:locY}) and (\ref{eq:nuLdef}).
\end{proof}

\begin{prop}\label{prop:symmetric}
Let $\phi,\psi\in \cP(X_T^G)$. The inner product between the normalized pseudo-Eisenstein series satisfy
\begin{equation}\label{eq:+sym}
(\theta_{\phi,0},\theta_{\psi,0}) = \frac{1}{|W|}\sum_{L\in W\backslash\cL}  \int_{L^{\textup{temp}}} r\check{r} A_0(\psi)(A_0(\phi^-))\circ(-\id))d\nu_{WL}.
\end{equation}
\end{prop}

\begin{proof}
Using Lemma \ref{lem:easyRewriting} and (\ref{eq:pushfwd}) with $f = (\check{r}\psi)\Sigma_W((\check{r}\phi)^-)$, from (\ref{eq:innernorm}) and (\ref{eq:Rrfunc}) we obtain
\[
(\theta_{\phi,0},\theta_{\psi,0}) =\frac{1}{|W|}\sum_{L\in\cL}  \int_{L^{\textup{temp}}} r\check{r} A_0(\psi)(A_0(\phi^-)\circ(-\id))d\nu_L,
\]
an the result follows from (\ref{eq:nuLdef}).
\end{proof}

\begin{defn}
For each $L\in W\backslash\cL$, define the measure $\mu_{WL}$  on $WL^{\textup{temp}}$ by
\begin{equation}\label{eq:Meas}
d\mu_{WL}(\lambda) := \left(\frac{|W|}{r(-\lambda,\overline{\chi})r(\lambda,\chi)}\right)d\nu_{WL}(\lambda).
\end{equation}
with $\nu_{WL}$ be as in (\ref{eq:nuLdef}).
\end{defn}

\begin{thm}\label{thm:Theorem}
For each $L\in W\backslash\cL$, the measure $\mu_{WL}$ is positive. Furthermore, the bilinear expression
\begin{equation*}
\langle\phi,\psi\rangle_{WL}:= \int_{WL^{\textup{temp}}}
\overline{A_0(r\check{r}\phi)} A_0(r\check{r}\psi)d\mu_{WL}
\end{equation*}
defines a positive semidefinite Hermitian form on the space of Paley-Wiener
functions on $V_\bbC$. The radical of this pairing consists of Paley-Wiener
functions $\phi$
for which we have $A_0(r\check{r}\phi)|_{WL^{\textup{temp}}}=0$.
We have a continuous map, isometric with respect to $\langle\phi,\psi\rangle_{WL}$ and with dense image
$A_{WL}:\cP(V_\bbC)\to L^2(WL^{\textup{temp}},\mu_{WL})^{W}$ given by $\phi\mapsto {A_0}(r\check{r}\phi)|_{WL^{\textup{temp}}}$. Finally we have
\begin{equation*}
(\theta_{\phi,0},\theta_{\psi,0})=
\sum_{L\in W\backslash\cL}\langle\phi,\psi\rangle_{WL}.
\end{equation*}
\end{thm}

\begin{proof}
Let us first consider the growth behaviour on $L^{\textup{temp}}$ of each factor of (\ref{eq:+sym}). Each summand of the integrand is given by the product of
certain Paley-Wiener functions on $L^{\textup{temp}}$ times factors which are of moderate growth on $L^{\textup{temp}}$ by
the analytical properties of $\rho_\alpha(s,\chi)$ which were discussed in Subsection \ref{sub:Rewriting}.

Next, note that the integrand of each factor of (\ref{eq:+sym}) is Hermitian and regular. Indeed, interchanging the roles of $\phi$ and $\psi$ results in complex conjugation of the restriction to $L^{\textup{temp}}$. 

To conclude the proof, we must show that the measure $\mu_{WL}$ defined in (\ref{eq:Meas}) is positive. We remark that, by $ W$-invariance, the expression $r(\lambda,\chi)r(-\lambda,\overline{\chi})$ is indeed real-valued. Further, given a residual subspace $L$, we can assume that this $L$ is chosen so that $\rR_{\chi,L}$ is a parabolic root subsystem. 
In this case, if $\lambda=c_L+i\mu\in L^{\textup{temp}}$
and if $w_L\in W({\rR_{\chi,L}})$ denotes the longest element
in the Weyl group $W({\rR_{\chi,L}})$ then $w_L(c_L)=-c_L$ and
$w_L\mu=\mu$. In other words, for $\lambda\in L^{\textup{temp}}$ we have
\begin{equation}\label{eq:conj}
-w_L\lambda=\overline{\lambda}.
\end{equation}
It follows that, for all $\lambda\in L^{\textup{temp}}$,
\[r(\lambda,\chi)r(-\lambda,\overline{\chi}) = C\times\prod_{\alpha\in\rR_{\chi,+}\setminus \rR_{\chi,L}}|\rho_\alpha(\langle\ra,\lambda\rangle,\chi)|^2,\]
where $C$ is the product of all constant values $\rho_\alpha(\langle\ra,c_L\rangle,\chi)$, for $\ra\in \rR_{\chi,L}$. It follows that the positivity of $r(\lambda,\chi)r(-\lambda,\overline{\chi})$ (and consequently of $\mu_{WL}$) reduces to the case in which $L$ is a residual point. Hence, we can assume, by induction, that the measures thus defined are positive for all orbits of residual subspaces with $\dim L>0$. Define
\[\cP(V_\bbC)^{\textup{ds}}:=\{\phi\in \cP(V_\bbC)\mid A_0(r\check{r}\phi)(c) = 0, \textup{ for all } c\textup{ residual}\}.\]
From the inductive hypothesis and the methods presented so far, it follows that the maps $\phi\mapsto A_{WL}(\phi)$ yield an isometry
\[L^2(G(F)\backslash G(\bbA),\xi)\supseteq \overline{\textup{span}\{\theta_\phi\mid\phi\in \cP(V_\bbC)^{\textup{ds}}\}}\cong \bigoplus_{\substack{L\in {}^rW\backslash \cL\\\dim L>0}} L^2(WL^{\textup{temp}},d\mu_{WL})^W.\]
Hence, given $\epsilon>0$ and $\psi\in \cP(V_\bbC)$, there exists $\phi\in \cP(V_\bbC)^{\textup{ds}}$ for which the $L^2$-norm
\begin{equation}\label{eq:approx}
\left\|\left({A_0}(r\check{r}\phi) - {A_0}(r\check{r}\psi)\right)\big|_{WL^{\textup{temp}}}\right\|<\epsilon,
\end{equation}
for all residual $L$ of positive dimension. Fix a residual point $c$ and let $\psi$ be a Paley-Wiener function for which ${A_0}(r\check{r}\psi)(c)\neq 0$ and $ {A_0}(r\check{r}\psi)(c') =0$ for $c'\notin {}^rWc$. 
From (\ref{eq:approx}) we can assume further that $\|A_{WL}(\psi)\|$ is very small on each orbit of residual subspace $L$ of positive dimension. 
If it were the case that the automorphic measure was negative on $Wc$, we would have, from Proposition \ref{prop:symmetric}, the induction hypothesis and the assumption on $\psi$ that
\[0\leq (\theta_{\psi,0},\theta_{\psi,0}) = \sum_{w\in W}\|{A_0}(r\check{r}\psi)(wc)\|^2\mu_{WL}(wc) + \sum_{\substack{L\in W\backslash \cL\\\dim L>0}}\langle\psi,\psi\rangle<0,\]
which is a contradiction.
\end{proof}

\subsection{Arthur Parameters and residual subspaces}\label{sub:APvsRS}

In this section, we link the the set of orbits of residual subspaces for the action of $W = W_{\cO(\chi)}$, with the set of Arthur parameters of Definition \ref{d:APspace}.

\begin{prop}\label{prop:ResvsAP}
The space $\ap$ is in bijection with $\cup\{WL^{\textup{temp}}\mid L\in W\backslash \cL\}\subseteq X_T^G$.
\end{prop}
\begin{proof}
For each $\psi\in \preap$, let $\psi'$ be in the conjugacy class $[\psi]$ such that $\psi'|_{W_F}=\varphi_{\chi,\nu}$ with $\nu$ unitary and that $\rho' = \psi'|_{SL(2,\bbC)}$ maps the torus of diagonal matrices of $SL(2,\bbC)$ inside $\dT$. Define $D:\ap \to W\backslash \fa^{G*}$ via
\[
D([\psi]) = \nu + d\rho'\left(
    \begin{pmatrix}
        1/2 & 0 \\ 0 & -1/2
    \end{pmatrix}
    \right).
\]
Conversely, given $\lambda = c_L + i\nu \in L^{\textup{temp}}$ for some $L\in\cL$, choose a $W$-conjugate of $L$ so that $c_L$ is dominant and thus $2c_L$ is the weighted Dynkin diagram of a nilpotent orbit of $\dHc = C_{\dG}(\varphi_\chi)^\circ$. The center determines $SL(2,\bbC)\to\dG$ and we let $\nu$ determine the homomorphism $W_F\to{}^LG$.
\end{proof}


With the identification of Proposition \ref{prop:ResvsAP} at hand, we let $\Xi :=  \cup\{WL^{\textup{temp}}\mid L\in W\backslash \cL\}$ and define \index{Measures!$\mu_0$}
\begin{equation}\label{eq:ParameterMeas}
\mu_0 := \sum_{L\in W\backslash\cL}\mu_{WL}.
\end{equation}
Using \ref{thm:Theorem} and the bijection of Proposition \ref{prop:ResvsAP}, it follows that there exists an isometry between $L^2(G(F)\backslash G(\bbA_F),\xi)_{[T,\cO(\chi)],0}^K$ and $L^2\left(\Xi,\mu_0\right)$. The next task in order to obtain our main result is to interpret this isometry by means of an integral transform. This is addressed in the next section.

\subsection{Finalizing the proof of Theorem \ref{thm:LowerBound}}
With the notions of normalized Eisenstein series $\cE_0(\nu,\chi,g)$\index{$\cE(\nu,\chi,g)$!$\cE_0(\nu,\chi,g)$} and normalized pseudo-Eisenstein series $\theta_{\phi,0}$ at hand, we shall argue next that the integral transform
$\fF:\Ltwoxi_{[T,\cO(\chi)],0}^K\to L^2(\Xi,\mu_0)$ \index{$\fF$}
\[
\fF(f)(\nu) = \int_{\aq} f(g) \cE_0(-\nu,\overline{\chi},g) dg
\]
is an isometric isomorphism of Hilbert spaces.
\begin{lem}\label{lem:L2IP1}
Given $\phi\in \cP(X_T^G)$ define
\begin{equation*}
u_\phi(\nu) := \int_{G(F)Z_{G}(\bbA_F)\backslash G(\bbA_F)}\overline{\theta_\phi(g)}\cE(\nu,\chi,g)  dg
\end{equation*}
This integral converges for $\re(\nu) = \nu_0\gg 0$ and defines a holomorphic function in its domain of convergence that satisfies
\[
u_\phi(\nu)=c(-\nu)^{-1}\left(\Sigma_W(\phi^-)(\nu)\right).
\]
\end{lem}

\begin{proof}
Given $\psi\in \cP(X_T^G)$, one has
\begin{multline*}
\int_{\re(\nu)=\nu_0\gg 0}\psi(\lambda)\int_{G(F)Z_{G}(\bbA_F)\backslash G(\bbA_F)}\overline{\theta_\phi(g)}\cE(\nu,\chi,g)  dg\;d\nu
\\ = \int_{\re(\nu)=\nu_0\gg 0} \left(\Sigma_W(\phi^-)(\nu)\right)\psi(\nu)\frac{d\nu}{c(-\nu)},
\end{multline*}
in which the exchange of integrals is allowed by the estimates used in the proof of \cite[Proposition II.1.10]{MW2} and use was made of equations (\ref{eq:pseudo-Eisenstein}) and (\ref{eq:inner1}). Since this holds for all $\psi\in \cP(X_T^G)$ the result follows.
\end{proof}

\begin{lem}\label{lem:L2IP2}
Given $\phi\in \cP(X_T^G),$ it holds that
\begin{equation*}
\int_{G(F)Z_{G}(\bbA_F)\backslash G(\bbA_F)} \overline{\theta_{\phi,0}(g)}\cE_0(\nu,\chi,g) dg = A_0^-(r\check{r}(\phi^{-}\circ\id))(\nu).
\end{equation*}
\end{lem}

\begin{proof}
Recall that $\theta_{\phi,0}(g) = \theta_{\check{r}\phi}$.
Using Lemmas \ref{lem:NormEis}, \ref{lem:L2IP1} and (\ref{eq:R-NF}), we get
\begin{eqnarray*}
\int_{G(F)Z_{G}(\bbA_F)\backslash G(\bbA_F)} \overline{\theta_{\phi,0}(g)}\cE_0(\nu,\chi,g) dg
&=& \frac{1}{|W|}c(-\nu)r(\nu,\chi)r(-\nu,{\overline\chi})u_\phi(\nu)\\
&=& \frac{1}{|W|}\sum_{w\in W}c(-w\nu)r(-w\nu,w{\overline \chi})\phi^-(-w\nu),
\end{eqnarray*}
where we used in the last line that $\frac{r(\nu,\chi)}{r(w\nu,w\chi)} = \frac{r(-w\nu,w{\overline \chi})}{r(-\nu,{\overline \chi})}$. This proves the Lemma.
\end{proof}

Using that $-\overline{\nu}\in W\nu$ (see (\ref{eq:conj})) so that $ A_0^-(r\check{r}(\phi^{-}\circ\id))(\nu) = \overline{A_0(r\check{r}\phi)(\nu)}$ and $\overline{\cE_0(\nu,\chi,g)} = \cE_0(-\nu,\overline{\chi},g)$ for all $\nu$ in the support $\cup_{L\in W\backslash \cL} WL^{\textup{temp}}$, it follows from Lemma \ref{lem:L2IP2} that $\cF(\theta_{\phi,0})(\nu) = A_0(r\check{r}\phi)(\nu)$ and hence the isometry part of Theorem \ref{thm:LowerBound} follows immediately from Theorem \ref{thm:Theorem}.

Finally, for $v$ any local place, let $A_v$\index{$G$!$A_v$} denote the maximal $F_v$-split torus in the local group $G_v$. Denote by $\cH(G_v,K_v)$\index{$\cH(G(\bbA_F),K)$!$\cH(G_v,K_v)$} the algebra of compactly supported functions on $K_v\backslash G_v/K_v$. We let $S_v:\cH(G_v,K_v)\to \bbC[X_*(A_v)]^{W}$ denote the Satake isomorphism and the $\ast$-structure of $\cH(G_v,K_v)$ can be described via
\begin{equation}\label{eq:Star1}
S_v(h^*)(\tau) = \overline{S_v(h)(\overline{\tau^{-1}})},
\end{equation}
for every $\tau\in\hat{A}_v$. For $v$ Archimedean $\cH(G_v,K_v)$ is the subalgebra of left and right invariant elements of the corresponding Archimedean Hecke algebra (the Archimedean Hecke algebra is described in \cite[Paragraph 3]{F} or \cite[1.1]{BJ} although our notation differs from the one in \cite{BJ}). We let $S_v:\cH(G_v,K_v)\to \textup{Sym}[\bbC \otimes X_*(A_v)]^{W}$ denote now the Harish-Chandra isomorphism and one can describe the $\ast$-structure similarly as
\begin{equation}\label{eq:Star2}
S_v(h^*)(\lambda) = \overline{S_v(h)(-\overline{\lambda})},
\end{equation}
for every $\lambda\in\Lie(\hat{A}_v)$. Now, for each $\nu\in X_T^G$ the local factors of the twisted automorphic character $\chi_\nu$ yields characters of $\cH(G_v,K_v)$ by means of $S_v$, and thus a character of $\cH(G(\bbA), K )$ given by
\[\chi_\nu(h) = \prod_vS_v(h_v)(\chi_{v,\nu}),\]
for any element of the form $h = \otimes_v h_v\in \cH(G(\bbA), K )$. A straightforward computation yields that, if $\nu$ is such that $-\overline{\nu} \in W\nu$ (see(\ref{eq:conj})), then
\begin{equation}\label{eq:star}
\omega_\nu(h^*) = \overline{\omega_\nu(h)},
\end{equation}
for all $h\in\cH(G(\bbA), K )$. Now, it is known that the Eisenstein series $\cE(\nu,\chi,g)$ is an eigenfunction for the convolution action of $\cH(G_v,K_v)$ with eigenvalue $\chi_{v,\nu}$ for each $v$. It then follows that, for $h_v\in \cH(G_v,K_v)$ and for $\nu\in \Xi \cong\ap$, we have
\begin{eqnarray*}
\fF(h_v\cdot f)(\nu) &=& \int_{G(F)Z_{G}(\bbA_F)\backslash G(\bbA_F)}\cE_0(-\nu,{\overline \chi},g)(h_v\cdot f)(g) dg\\
&=& \int_{G(F)Z_{G}(\bbA_F)\backslash G(\bbA_F)} \overline{(h_v^*\cdot \cE_0(\nu,\chi,g))}f(g) dg\\
&=& \chi_{v,\nu}(h_v)\fF(f)(\nu),
\end{eqnarray*}
proving the equivariance of $\fF$.

\section{The case $[T,\cO(1)]$ for $G$ split: The cascade of contour-shifts}\label{s:Cascade}
From here onward we will concentrate on proving Theorem \ref{thm:MainTheorem}. 
The method of proof owes a lot to the proof in the classical cases by Moeglin \cite{M1}. We use induction to the rank of $G$, and for a suitable 
maximal proper standard Levi subgroup $G'\subset G$ we define a combinatorial structure $C$, called a ``cascade'', consisting of a 
collection of segments in pole spaces with the purpose of  
prescribing a cascade of 
successive contour shifts which result in rewriting the 
inner product formula of two pseudo Eisenstein series $\theta_\phi, \theta_\psi$ as given in Corollary \ref{cor:NF.InProd} as a sum of 
integrals of iterated residues at various pole spaces $L$ over 
a contour of the form $z_L+iV^L$, where (and this is the main point) 
$z_L$ is close to the center $c_L$ of $L$. The cascade $C$ needs to 
possess several favorable properties in order that when performing 
these successive contour shifts, no residues of the integrands caused 
by the critical zeroes of $\rho$ will be picked up. Cascades $C$ 
with such favorable properties are easy to define for 
all classical cases, but for the exceptional cases this was done with 
the help of a computer in \cite{DMHO}. 

A cascade $C$ for $R'\subset R$ consists collection of pairs $(\sigma,L)$ with $L$ a $\Omega$-pole space and $\sigma\subset L$ a segment. Here $\Omega$\index{$\omega$!$\Omega$} is a rational $(r,0)$-form on $V_\bbC$ given by 
\begin{equation}
\Omega(\lambda):=\frac{d\lambda}{c'(\lambda)c(-\lambda)}
\end{equation}
Observe that $\Omega$ is $W'$-invariant, in fact it is the $W'$-average 
of $\omega$. 
For its definition and properties we refer to \cite[Appendix 6]{DMHO}. 
The cascade $C$ comes equipped with several additional data. The 
set $\cCO^C$\index{$\cCO^C$} denotes a set of representatives of the $W'$-orbits 
of $\Omega$-pole spaces $L$ which appear in $C$, i.e. for which 
there exists a segment $\sigma\in L$ such that $(\sigma,L)\in C$.

For our present purpose it is enough to recall the main result of 
\cite{DMHO}: 
\begin{thm}\label{thm:mainC}
Let $\rho$ an entire function on $\bbC$ 
such that $\Lambda:=s^{-1}(1-s)^{-1}\rho$ satisfies the properties 
mentioned in Proposition \ref{prop:analyticprops}\footnote{We do not need 
that $\Lambda$ is the completed Dedekind zeta function $\Lambda_F$ of a 
number field.}. 
Choose constants $\fR,\fT,\fT'>0$ with $\fT>\fR$ and $\fT'>3\fT+2\fR^2$ as in \cite[V.2.2]{MW2}. Let $\phi,\psi\in\cP^\fR(V_\bbC)$.
There exists a maximal proper standard Levi subgroup $G'\subset G$ with 
root systems $R'\subset R$, and a Cascade $C$ 
for $R'\subset R$ (see \cite[Appendix 6]{DMHO} satisfying the following
property: Let $\cCO^{C,presym}$\index{$\cCO^C$!$\cCO^{C,presym}$} be collection of pole spaces 
in $\cCO^C$ after removing the $M$ such that $\textup{Ord}_M(\Omega)=0$ and $M^{\textup{temp}}\not\subset\cup_{N\in \cL(\Omega_r)}N^{\textup{temp}}$.
For every pole space $M\in \cCO^{C,presym}$ we can choose a $\fT'$-general point $z_M$ close to the center $c_M$ of $M$ such that   
\begin{equation}\label{eq:nearcenter}
(\theta_\phi,\theta_\psi)=_{\fT}\sum_{M\in \cCO^{C,presym}}\int_{(z_M+iV^M)_{\leq \fT}}\sum_{\cF\in \cF(M)}C_{M,\cF}
\textup{Res}_{M,\cF}((\Sigma_{W}(\phi^-))(\Sigma'_{W'}(\psi))\Omega)
\end{equation}
for suitable constants $C_{M,\cF}$. 
\end{thm}
This is the starting point for our proof of the main result Theorem \ref{eq:norm=all} in the next section.

\section{Proof of Theorem \ref{thm:MainTheorem}: Symmetrization and comparison with 
residue distributions}\label{s:MainResult}
We can define the inner product formula (\ref{eq:inprodpseueisNF}) for more general functions $\Lambda$ as in Theorem \ref{thm:mainC}. The next result holds true in this level of generality.

\begin{thm}\label{t:MainResult2}
Let $q_{\fT}$ denote the orthogonal 
projection onto $(L^{2,K}_{[T,\cO(1)]})_{[-\fT,\fR]}\subset L^{2,K}_{[T,\cO(1)]}$. 
For all $\phi, \psi\in\cP^\fR(V_\bbC)$ we have: 
\begin{align}\label{eq:truncmain}
(\theta_\phi,&q_{\fT}(\theta_\psi))\\
\nonumber&=
\sum_{L\in W\backslash\cL}|W|\int_{L^{\textup{temp}}_{\leq \fT}}(r(-\lambda)r(\lambda))^{-1}
{A_0}(r\psi)(\lambda)\overline{{A_0}(r\phi)(\lambda)}d\nu_L(\lambda),
\end{align}
where the left-hand side of (\ref{eq:truncmain}) is interpreted for general functions $\Lambda$, as above.
\end{thm}
\begin{proof}
We use a comparison to the computation of $X_{V,\lambda_0}(\theta)$ for $\theta\in \cP^\fR(V_\bbC)$ (with $\fR$ as in Theorem \ref{thm:mainC}) 
in terms of residue distribution. We know the answer in this case, since by the uniqueness \cite[Lemma 3.1]{HO1}, the local tempered distributions representing $X_{V,\lambda_0}$ on the tempered pole spaces $M^{\textup{temp}}$ can be expressed in terms of the iterated residues 
prescribed by the cascade $C$. This will yield the 
required information about the iterated residues. 
Using Theorem \ref{thm:mainC} and computations as in Lemma \ref{lem:easyRewriting} (see also \cite[Proposition 3.1]{HO1}), we have for all $\theta\in\cP^\fR(V_\bbC)$ that: 
\begin{align}\label{eq:compare}
X_{V,\lambda_0}(\theta)
&=\sum_{M\in \cCO^{C,presym}}\int_{z_M+iV^M}
C_{\cF,M}\sum_{\cF\in \cF(M)}\textup{Res}_{M,\cF}(|W'|A'_0(\theta)\Omega)\\
\nonumber&=\sum_{c\in \mathcal{C}^+}\frac{|W|}{|W_c|}Y_c(A_0(\theta))\\
\nonumber&=\sum_{c\in \mathcal{C}^+}Y_{Wc}(A_0(\theta))\\
\nonumber&=Y(A_0(\theta))
\end{align}
where for each dominant center $c$ we have (see (\ref{eq:locY})) $Y_c=\sum_{\{L\in \cL\;|\;c_L=c\}}Y_L$ and $Y_L$ is the push forward (\ref{eq:pushfwd}) of a positive smooth $N_W(L)$-invariant 
measure ($N_W(L)$ is the normalizer in $W$ of the space $L$) which extends holomorphically to a tubular neighborhood of $L^{\textup{temp}}\subset L$. 
Moreover, the $Y_{Wc}$\index{$X_{V,\lambda_0}$!$Y_{Wc}$} are defined as integration against $d\nu_{WL}$ (see (\ref{eq:nuLdef})) and 
$Y:=\sum Y_{Wc}$\index{$X_{V,\lambda_0}$!$Y$}.  The resulting sum of push forwards of 
smooth positive measures $\sum_{c\in \mathcal{C}^+}Y_{Wc}$ is a $W$-invariant functional on $\cP^\fR(V_\bbC)$, which is a distribution $Y_c$ of order zero on each real subspace of the form $c+iV$. 

Both the LHS and the RHS are sums of tempered distributions which are normal derivatives of boundary value distributions on $L^{\textup{temp}}\subset L_\bbC$ where $L$ is a pole space such that $c_L=c$ 
see \cite{HO1}.
What we want to establish is that the kernels defining the LHS and the RHS of (\ref{eq:compare}) are equal.
We can only conclude something like that if we have a unique representation of the functional (\ref{eq:compare}) 
by these kernels. In order to have a unique representation we need to symmetrize. 

The cascade $C$ of contour shifts comes with a collection of representatives $\cCO^C$ of the $W'$-orbits of $\Omega$-pole spaces in $C$ 
(cf. \cite[Appendix 6]{DMHO}), and for 
each $L\in\cCO^C$ a pair $(L_0,d)$ 
consisting of a standard $\omega$ pole space $L_0$ and
$d\in W$ of minimal length in $W'd$ such that 
$L=dL_0$, together with a choice of elements $w_i=u_id\in W'd$ such that $\{L^i=w_i(L_0)\}$
runs over the collection of $\Omega$ pole spaces in 
$C\cap W'L$. 
We choose a set $\cC(\omega)$ of standard $\omega$-pole spaces which form a set of representatives of the $W$-orbits in $\cCO^C$.

On each $L^i_{\bbC}$ we have the meromorphic $(\textup{dim}(L^i),0)$-form 
on $L^i$ defined by taking the sum of the iterated residues of $A'_0(\theta)\Omega$ on $L^i$ along the flags of pole 
spaces ending in $L^i$ (cf. 
\cite[Corollary 5.4]{DMHO}):  
\begin{equation}
\sum_{\cF\in\cF(L^i)}C_{\cF,L^i}\textup{Res}_{\cF,L^i}(A_0'(\theta)\Omega)
\end{equation} 
where $\cF(L^i)$ denotes the collection of flags of pole spaces in the cascade $C$ which contain $L^i$. 
We form a kernel $K_{L_0}$, a meromorphic function on $L_{0,\bbC}$, by 
\begin{equation}\label{eq:defK}
K_{L_0}(\theta)d^{L_{0,\bbC}}(\lambda):=\sum_i \sum_{\cF\in\cF(L^i)}C_{\cF,L^i}w_i^*(\textup{Res}_{\cF,L^i}(A'_0(\theta)\Omega))
\end{equation} 
This kernel satisfies: 
\begin{lem}\label{lem:kernX}
\begin{enumerate}
\item[(i)] For all entire functions $\theta$ on $B_\fR=\{\lambda\in V_\bbC\mid |\textup{Re}(\lambda)|<\fR\}$,   
$K_{L_0}(\theta)$ is a meromorphic function on $L_{0,\bbC}^\fR = L_{0,\bbC} \cap B_\fR$.
\item[(ii)] There exists a decomposition $K_{L_0}(\theta)=\sum_{M\in WL_0}K^M(\theta)$ 
where for each $M\in WL_0$,  $K^M(\theta)$ is a meromorphic 
function on $L_{0,\bbC}^{\fR}$, and with the kernel $K^M$ supported on $M$ in the sense that 
$K^M(\theta)=0$ if $\theta$ vanishes on $M$ of high order. There exists a $W$-invariant polynomial $S_M$,  
independent of $\theta$, such that $(S_M|_{M_\bbC})K^M(\theta)\in \cP^{\fR}(L_\bbC)$ for all $\theta\in \cP^\fR(V_\bbC)$.
\item[(iii)] Let  $H^{\fR}$ be the algebra of holomorphic $W$-invariant 
functions on $B_\fR$ of moderate growth in the imaginary direction (the multiplier algebra of $\cP^{\fR}(V_\bbC)^W$). 
Suppose that $n\in\bbZ_{\geq 0}$ 
is the order of the kernel $K^M$ on $M$, i.e. the minimal integer such that $K^M(\theta)$ vanishes 
for all $\theta\in \cP^\fR(V_\bbC)$ which vanish on $M$ to order $n+1$. Suppose that  $\theta\in \cP^\fR(V_\bbC)$ has vanishing order 
$n$ on $M$. Then for all $f\in H^{\fR}$ we have, for all $M\in WL_0$, that $K_{L_0}^M(f\theta)=f|_MK^M(\theta)$. 
\end{enumerate}
\end{lem}
\begin{proof}
We have (cf. \cite[equation (16)]{DMHO}) for $\lambda\in L_{0,\bbC}$ that: 
\begin{equation}\label{eq:resM}
w_j^*(\textup{Res}^M_{\cF,L^j}(A'_0(\theta)\omega)):=\sum_i D_{L^j,\cF,i}(\sum_{u\in W': ud_j(L_0)=M}c'(ud_j\lambda)\theta(ud_j\lambda)|_{L_\bbC}\omega_{L,\cF,i})
\end{equation}
which easily shows (i) and (ii) in view of (\ref{eq:defK}). 
For (iii) note that $w_j^*(\textup{Res}^M_{\cF,L^j}(A'_0(\theta)\omega))$ satisfies \cite[eq. (1) in proof of Lemma V.2.10]{MW2}. 
From this we see that \cite[Lemma V.2.10]{MW2} applies, and this yields (iii).
\end{proof}
As a result we obtain 
\begin{equation}\label{eq:Xpresymm}
X_{V,\lambda_0}(\theta)=
\sum_{M\in \cC(\omega)}\int_{z_M+iV^M}K^M(\theta)(\lambda)d^M(\lambda)
\end{equation}
We can now symmetrise for each $\omega$-pole space $M$ over the normalizer $N_W(M)$. Indeed, we move the contours $z_M +iV^M$ to 
the contours $wz_M+iV^M$ in the orbit under $N_W(M)$ and then pull back to $z_M +iV^M$ with $w^*$, and take the average. This procedure 
introduces possible residues at pole spaces $L$ of codimension one in $M$, but clearly only those $L$ are involved for which $c_L=c_M$. 
At these pole spaces $L$ we symmetrise similarly over $N_W(L)$. Finally this results in a rewriting of (\ref{eq:Xpresymm}) 
in the form 
\begin{equation}\label{eq:Xsymm}
X_{V,\lambda_0}(\theta)=
\sum_{M\in \cC(\omega)}\int_{z_M+iV^M}K_X^{M,sym}(\theta)(\lambda)d^M(\lambda)
\end{equation}
where for all $\Omega$-pole spaces $M$, the meromorphic kernel $K_X^{M,sym}$ is invariant under $N_W(M)$ and clearly also 
satisfy Lemma \ref{lem:kernX}. 

Now we look at the right hand side $Y(A_0(\theta))$. In a similar fashion, but easier since the symmetrisation step over $N_W(M)$ 
is not necessary (the kernels are already $N_W(M)$-invariant now) we can write:
\begin{equation}\label{eq:Ysymm}
Y(A_0(\theta))=
\sum_{M\in \cC(\omega)}\int_{z_M+iV^M}K_Y^{M,sym}(\theta)(\lambda)d^M(\lambda)
\end{equation}
where each kernel $K_Y^{M,sym}(\theta)$ is $N_W(M)$-invariant. 

We claim now that for all pole spaces $M$ we have 
\begin{equation}\label{eq:Xsymm=Ysymm}
K_X^{M,sym}(\theta)=K_Y^{M,sym}(\theta)
\end{equation}
Indeed, suppose that we have a collection of meromorphic kernels $K_Z^{M,sym}(\theta)$ on the $\Omega$-pole spaces 
which satisfy Lemma \ref{lem:kernX}, are $N_W(M)$-invariant, and such that for all $\theta\in \cP^\fR(V_\bbC)$: 
\begin{equation}\label{eq:0symm}
0=\sum_{M\in \cC(\omega)}\int_{z_M+iV^M}K_Z^{M,sym}(\theta)(\lambda)d^M(\lambda)
\end{equation}
Suppose that \emph{not} all kernels $K_Z^{M,sym}$ are identically $0$. Then take $M$ of maximal dimension such that 
$K_Z^{M,sym}$ is not identically $0$, and take $\theta\in \cP^\fR(V_\bbC)$ such that it is of maximal vanishing order at $M$ 
while $K_Z^{M,sym}(\theta)\not=0$. 

Next for any $n\in\bbN$ we construct a $N_W(M)$-invariant polynomial $Q^M_n$ 
which is nonzero on $M$, which is zero of order at least $n$ at all pole spaces $N$ with $\textup{dim}(N)\leq \textup{dim}(M)$ 
and $N\not=M$, and such that $D(Q^M_n)|_M=0$ for all constant coefficients differential operators $D\in S(V_M)_+$ 
of degree less than $n$. First for each pole space $N$ with $\textup{dim}(N)\leq \textup{dim}(M)$ 
and $N\not=M$, and each $n\in \bbN$, we can find a polynomial $Q^M_{n,N}$ which is nonzero on $M$, such that 
$D(Q^M_{n,N})|_M=0$ for all constant coefficients differential operators $D\in S(V_M)_+$ of degree less than $n$, 
and vanishing on $N$ of order at least $n$. A construction is given in the proof of \cite[Lemma V.2.11]{MW2} (where this is 
applied to $M$ and $N=wM\not=M$, but the same construction works for any pole space $N\not=M$ such that 
$\textup{dim}(N)\leq \textup{dim}(M)$). 
Finally define $Q^M_n=\prod_{w\in N_W(M)} w^*(\prod_N Q^M_{N,n})$ 
where $N$ runs over the set of pole spaces with $\textup{dim}(N)\leq \textup{dim}(M)$ 
and $N\not = M$. This clearly has the properties we stated above. 

Using Lemma \ref{lem:kernX}(iii) and the choice of $\theta$ we see that it follows that:  
\begin{equation}
K^{M,sym}_Z(Q^M_n\theta)=Q^M_n|_MK^{M,sym}_Z(\theta)\not=0
\end{equation}
while for all $\omega$-pole spaces $N\not=M$ 
\begin{equation}
K^{N,sym}_Z(Q^M_n\theta)=0.
\end{equation}
(Indeed if $\textup{dim}(N)> \textup{dim}(M)$ then this follows from the definition of $M$, 
while for  $\textup{dim}(N)\leq \textup{dim}(M)$ with $N\not=M$ this follows from Lemma \ref{lem:kernX}(ii) 
and the properties of $Q^M_n$.)

Next observe that, due to the choice of $\theta$, for all $f\in H^\fR$ we obtain that $K^{M,sym}_Z(f\theta)=(f|_M)K^{M,sym}_Z(\theta)$.
By \cite[V.2.11]{MW2}(iii) there exists a polynomial $P_M$ on $M$ such that the set of restrictions $f|_M$ with 
$f\in H^\fR$ contains the subspace $(P_M|_M)\bbC[M_\bbC]^{N_W(M)}$ of $\bbC[M_\bbC]^{N_W(M)}$. By Lemma \ref{lem:kernX}(ii) 
there exists an $S_M\in \bbC[V_\bbC]^{N_W(M)}$ such that $S_M|_M$ is nonzero and for all $\theta\in \cP^\fR(V_\bbC)$:   
$K^{M,sym}_Z(S_M\theta)=(S_M|_M)K^{M,sym}_Z(\theta)\in \cP^{\fR}(M_\bbC)$. 

Altogether we thus obtain that 
\begin{equation}
0=\int_{c_M+iV^M}F(\lambda)(P_M|_M)(S_M|_M)K^{M,sym}_Z(Q^M_n\theta)(\lambda)d^M(\lambda)
\end{equation}
for all $F\in \bbC[M_\bbC]^{N_W(M)}$. Since $K^{M,sym}_Z(Q^M_nS_M\theta)(\lambda)\in \cP^{\fR}(M_\bbC)^{N_W(M)}$ and is nonzero, 
this is absurd.  This proves the assertion (\ref{eq:Xsymm=Ysymm}).

Now we return to (\ref{eq:nearcenter}). We follow the cascade of contour shifts as before, but this time
but this time with $\theta$ replaced by the meromorphic function $\psi\Sigma_W(\phi^-)$, and with the $\fT$-truncated integrals.
Because we have only moved the base points of order $0$ pole spaces inside their admissible sets, and because the positive order poles spaces $L_0\in \cC_0(\omega)$, the kernels $K_{L_0}$ 
have no denominators taking critical values at $c_{L_0}$ (cf.  \cite[Theorem 5.1, Theorem 5.2]{DMHO}), 
we obtain iterated residue integrals involving exactly the same kernels $K^{M,sym}_X(\theta)$ and $K^{M,sym}_Y(\theta)$ 
as before, but just with $\theta$ replaces by $\psi\Sigma_W(\phi^-)$. This yields: 
\begin{align}\label{eq:nearcenterA}
(\theta_\phi,\theta_\psi)&=_{\fT}\sum_{M\in \cCO^{C,presym}}\int_{(z_M+iV^M)_{\leq \fT}}
C_{\cF,M}\sum_{\cF\in \cF(M)}\textup{Res}_{M,\cF}(|W'|
A'_0(\psi\Sigma_W(\phi^-))\Omega)\\
\nonumber&=_{\fT}\sum_{M\in \cC(\omega)}\int_{(z_M+iV^M)_{\leq \fT}}K_X^{M,sym}(\psi\Sigma_W(\phi^-))(\lambda)d^M(\lambda)\\
\nonumber&=_{\fT}\sum_{M\in \cC(\omega)}
\int_{(z_M+iV^M)_{\leq \fT}}K_Y^{M,sym}(\psi\Sigma_W(\phi^-))(\lambda)d^M(\lambda)\\
\nonumber&=_{\fT}\sum_{c\in\cC_+}\frac{|W|}{|W_c|}\sum_{L\in\cL:c_L=c}\int_{(z_L+iV^L)_{\leq \fT}}A_0(\psi\Sigma_W(\phi^-))Y_{{\Sigma_{\chi,L}},c_L}(\{c_L\})d\mu^L(\lambda)\\
\nonumber&=_{\fT}\sum_{L\in W\backslash\cL}|W|\int_{(z_L+iV^L)_{\leq \fT}}(r(-\lambda)r(\lambda))^{-1}
{A_0}(r\phi^-)(-\lambda){A_0}(r\psi)(\lambda)d\nu_L(\lambda)\\
\nonumber&=_{\fT}\sum_{L\in W\backslash\cL}|W|\int_{L^{\textup{temp}}_{\leq \fT}}(r(-\lambda)r(\lambda)^{-1}
{A_0}(r\psi)(\lambda)\overline{{A_0}(r\phi)(\lambda)}d\nu_L(\lambda)
\end{align}
where in the  last line we used that for each $L\in\cL_+$, the kernel 
of the integral over $(z_L+iV^L)_{\leq \fT}$ is 
holomorphic in $L^{\textup{temp}}$. To see this, we first of all recall that on residual spaces $L\in\cL_+$, 
the roots which are constant on $L$ have integral values on $L$. For the non-constant roots involved in 
\begin{equation}
\frac{|W|}{r(-\lambda)}{A_0}(r\phi^-)(-\lambda)|_{L^{\textup{temp}}} = \Sigma_W(\phi^-)(\lambda)|_{L^{\textup{temp}}},
\end{equation}
applying \cite[Corollary 3.2, Theorem 3.7]{DMHO}
and the fact that on $L^{\textup{temp}}$ we have 
${A_0}(r\phi^-)(-w\lambda) = \overline{{A_0}(r\phi)(\lambda)}$ where $w\in W$ 
is such that $-w(c_L+i\mu)=c_L-i\mu$ if $\mu\in V^L$, we get that 
\begin{equation}
\frac{1}{r(\lambda)}{A_0}(r\psi)(\lambda)|_{L^{\textup{temp}}}
\end{equation}
is holomorphic on $L^{\textup{temp}}$ as well. 
Hence we can send all base points $z_L$ to the corresponding center 
$c_L$. 

At this point we can use the interpretation of the truncation of the integrals in terms of the orthogonal 
projection $q_{\fT}$ onto $(L^{2,K}_{[T,\cO(1)]})_{[-\fT,\fR]}\subset L^{2,K}_{[T,\cO(1)]}$
Equation (\ref{eq:nearcenterA}) and the above holomorphicity arguments 
show that the conditions of \cite[Lemma V.2.9]{MW2} are satisfied, so following its proof almost verbatim 
we obtain: 
\begin{align*}\label{eq:trunc}
(\theta_\phi,&q_{\fT}(\theta_\psi))\\
&=
\sum_{L\in W\backslash\cL}|W|\int_{L^{\textup{temp}}_{\leq \fT}}(r(-\lambda)r(\lambda))^{-1}
{A_0}(r\psi)(\lambda)\overline{{A_0}(r\phi)(\lambda)}d\nu_L(\lambda)
\end{align*}
\end{proof}
\subsection{Proof of main Theorem \ref{thm:MainTheorem}}
\begin{proof}
Taking the limit $\fT\to\infty$ in Theorem 
\ref{t:MainResult2} finishes the proof of the main Theorem \ref{thm:MainTheorem}.
\end{proof}

\section*{Acknowledgements}
During the period over which this project has taken place, we thank the financial supports given by the ERC-advanced
grant 268105; the grants of Agence Nationale de la Recherche with references ANR-08-BLAN-0259-02 and ANR-13-BS01-0012 FERPLAY; the Archimedes LabEx ANR-11-LABX-0033, the foundation A*MIDEX ANR-11-IDEX-0001-02, funded by the program ``Investissements d’Avenir” led by the ANR; the Engineering and Physical Sciences Research Council grant EP/N033922/1 (2016) and the special research fund (BOF) from Ghent University BOF20/PDO/058.

\printindex

\begin{thebibliography}{99}

\bibitem[A]{A} Arthur, J.,
{\it Unipotent automorphic representations: conjectures,}
in Orbites unipotentes et repr\'esentations, II.
Ast\'erisque {\bf 171-172} (1989), 13--71.

\bibitem[BC]{BC} Bala, P., Carter, R.W.,
{\it The classification of unipotent and nilpotent elements,}
Indagationes Mathematicae, {\bf 77}(1) (1974), 94--97.

\bibitem[Bu]{Bu} Bump, D.,
{\it Automorphic forms and representations,}
Cambridge Studies in Advanced Mathematics, {\bf 55}, Cambridge University Press, 1998.

 \bibitem[BJ]{BJ} Borel, A., Jaquet, H.,
 {\it Automorphic forms and automorphic representations,}
 in Automorphic forms, representations and {$L$}-functions ({P}roc. {S}ympos. {P}ure {M}ath., {O}regon {S}tate {U}niv., {C}orvallis, {O}re., 1977), {P}art 1, Amer. Math. soc., Providence, RI, 1979, pp. 189--207.

\bibitem[Car]{C} Carter, R.W., 
{\it Finite groups of Lie type,}
Wiley Classics Library, John Wiley
and sons, Chichester, UK, 1993.

\bibitem[Ca]{Ca} Casselman, .W.,
{\it The unramified principal series of p-adic groups. {I}. {T}he
spherical function,}
Compositio Math., {\bf 40}(3)(1980), 387--406.

 \bibitem[CKK]{CKK} Ciubotaru, D., Kato, M., Kato, S.,
 {\it On characters and formal degrees for classical and Hecke algebras,}
 Invent. Math. {\bf 187}(3) (2012), 589–-635.

\bibitem[Co]{Co} Cogdell, J., 
{\it Lectures on $L$-functions. Converse theorems and functoriality of $GL_n$,}
in Lectures on Automorphic $L$-functions, Fields Institute Monographs, Amer. Math. Soc.,
2004, pp. 3--96.

\bibitem[CO]{CO} Ciubotaru, D., Opdam, E., 
{\it A uniform classification of discrete series representations of affine Hecke algebras,}
Algebra Number Theory, {\bf 11} (2017), no. 5, 1089--1134.

\bibitem[DMHO]{DMHO} De Martino, M., Heiermann, V., Opdam, E.,
{\it Residue distributions, iterated residues, and the
spherical automorphic spectrum,}
arXiv:2207.06773 (2022).

 \bibitem[DMO]{DMO} De Martino, M., Opdam, E.,
 {\it Limit transition between the spherical spectrum of graded and affine Hecke algebras.}
 In preparation.

\bibitem[F]{F} Flath, D.,
{\it Decomposition of representations into tensor products,} 
in Automorphic forms, representations and {$L$}-functions ({P}roc. {S}ympos. {P}ure {M}ath., {O}regon {S}tate {U}niv., {C}orvallis, {O}re., 1977), {P}art 1, Amer. Math. soc., Providence, RI, 1979, pp. 179--183.

\bibitem[GL]{GL} Gelbart, S., Lapid, E. M.,
{\it Lower bounds for L-functions at the edge of the critical strip,}
Am. J. Math. {\bf 128} (2006), 619--638.

\bibitem[HC]{HC} Harish-Chandra,
{\it Automorphic forms on Semisimple Lie Groups,}
Lecture Notes in Mathematics, No. 62 Springer-Verlag, Berlin-New York, 1968.

 \bibitem[HO1]{HO1} Heckman, G.J., Opdam, E.M.,
 {\it Yang's system of particles and Hecke algebras,}
 Annals of mathematics {\bf 145} (1997), 139--173.

\bibitem[HO2]{HO2} Heckman, G.J., Opdam, E.M.,
{\it Harmonic analysis for affine {H}ecke algebras,}
in Current Developments in Mathematics (R. Bott, A. Jaffe, D. Jerison, G. Lusztig, I. Singer and S.-T. Yau, editors), Intern. Press, 1996, pp. 37--60.

\bibitem[H1]{H1} Heiermann, V.,
{\it D\'ecomposition spectrale et repr\'esentations sp\'eciales d'un groupe r\'eductif p-adique},
Journ. Inst. Math. Jussieu {\bf 3} (2004), 327--395.

\bibitem[H2]{H2} Heiermann, V.,
{\it Orbites unipotentes et p\^oles d'ordre maximal de la fonction $\mu$ de Harish-Chandra},
Canad. J. Math., {\bf 58} (2006), 1203--1228.

\bibitem[H3]{H3} Heiermann, V.,
{\it Op\'erateurs d’entrelacement et alg\`ebres de Hecke avec param\`etres d’un groupe r\'ductif p-adique: le cas des groupes classiques},
Sel. Math., {\bf 17} (2011), 713--756.

\bibitem[H4]{H4} Heiermann, V.,
 {\it  The value of the global intertwining operators on spherical vectors},
  Proc. Amer. Math. Soc. {\bf 147} (2019), 115--124.

\bibitem[HII]{HII} Hiraga, K., Ichino, A., Ikeda, T.,
 {\it Formal degrees and adjoint gamma factors (and errata)},
 J. Amer. Math. Soc., 21 (2008), no. 1, 283-304.

\bibitem[J]{J} Jacquet, H.
{\it On the residual spectrum of $GL(n)$,}
in Lie groups and Representations II, (College Park, Md., 1982/1983), Lecture Notes in Math. {\bf 1041}, Springer-Verlag,
New York (1984), 185--208.

\bibitem[K1]{K1} Kim, H.H.,
{\it The residual spectrum of $G_2$,}
Canad. J. Math., {\bf 48} (1996), 1245--1272.

\bibitem[K2]{K2} Kim, H.H.,
{\it Residual spectrum of odd orthogonal groups,}
Int. Math. Res. Notices, {\bf 17} (2001), 873--906.

\bibitem[KL]{KL} Kazhdan, D., Luzstig, G.,
{\it Proof of the Deligne-Langlands conjecture for Hecke Algebras,}
Invent. Math. {\bf 87} (1987), 153--215.

\bibitem[KO]{KO} Kazhdan, D., Okounkov, A.,
{\it On the unramified Eisenstein spectrum},
arXiv:2203.03486 (2022).


\bibitem[La1]{La1} Langlands, R.P.,
{\it On the functional equations satisfied by Eisenstein series,}
Lecture Notes in Mathematics \textbf{544}, Springer, 1976.

\bibitem[La2]{La2} Langlands, R.P.,
{\it Representations of abelian algebraic groups,}
Pacific Journal of Mathematics, \textbf{181}(3) (1997), 231--250.

\bibitem[Lu]{Lu} Lusztig, G.,
{\it Affine Hecke algebras and their graded version,}
J. Am. Math. Soc. \textbf{2}(3) (1989), 59--635.

\bibitem[M1]{M1} Moeglin, C.,
{\it Orbites unipotentes et spectre discret non ramifi\'e. Le cas des groupes classiques d\'eploy\'es,}
Compos. Math. \textbf{77}(1) (1991), 1--54.

\bibitem[M2]{M2} Moeglin, C.,
{\it Sur les formes automorphes de carr\'e int\'egrable,}
in Proceedings of the International Congress of Mathematicians, Vol. II (Kyoto, 1990),  815--819, Math. Soc. Japan, Tokyo, 1991.

\bibitem[Mi]{Mi} Miller, S. D.,
{\it Residual automorphic forms and spherical unitary representations of exceptional groups,}
Annals of mathematics {\bf 177} (2013), 1169--1179.

\bibitem[MW1]{MW1} Moeglin, C., Waldspurger, J.-L.,
{\it Le spectre r\'esiduel de $GL(n)$,}
Ann. Sci. \'Ecole Norm. Sup. {\bf 22} (1989), 605--674.

\bibitem[MW2]{MW2} Moeglin, C., Waldspurger, J.-L.,
{\it Spectral decomposition and Eisenstein series,}
Cambridge tracts in Mathematics \textbf{113}, Cambridge University Press, 1995.

 \bibitem[O1]{O-Sp} Opdam, E.M.,
 {\it On the spectral decomposition of affine Hecke algebras,}
 J.Inst. Math. Jussieu \textbf{3}(4) (2004), 531--648

 \bibitem[O2]{O-Supp} Opdam, E.M.,
 {\it The central support of the Plancherel measure of an affine Hecke algebra,}
 Moscow Mathematical Journal \textbf{7}(4) (2007), 723--741.

 \bibitem[O3]{O-STM} Opdam, E.M.,
 {\it Spectral transfer morphisms for unipotent affine Hecke algebras.}
 arXiv:1310.7790.

 \bibitem[Re2]{Re1} Reeder, M.,
 {\it Formal degrees and $L$-packets of unipotent discrete series representations of exceptional $p$-adic groups,}
 with an appendix by Frank Luebeck. Crelle's Journal {\bf 520} (2000), 37--93.

\bibitem[Re3]{Re2} Reeder, M.,
{\it Torsion automorphisms of simple Lie algebras,}
L'Enseignement Mathematique {\bf 56}(2) (2010), pages 3--47.

\bibitem[Si]{Si} Silberger, A.,
{Introduction to harmonic analysis on reductive p-adic groups} (Mathematical Notes No. 23), Princeton University Press, 1979.

\bibitem[Ta]{Ta} Tate, J.,
{\it Fourier Analysis in Number Fields and Hecke's Zeta-Functions.}
In: Algebraic Number Theory, eds.: Cassels, Frolich, Academic Press, 1967, pp. 305--347.

\bibitem[We]{We} Weil, A.,
{\it Basic number theory.}
Springer-Verlag, New York-Berlin, 3rd edition, 1974.

\end{thebibliography}
\end{document}